\documentclass{article}
\usepackage[utf8]{inputenc}
\usepackage[T1]{fontenc}
\usepackage{amsmath}
\usepackage{amssymb}
\usepackage{mathrsfs}
\usepackage{amsthm}
\usepackage{stmaryrd}
\usepackage{geometry}
\usepackage{mathtools}
\usepackage{biblatex}
\usepackage{dsfont}

\usepackage{hyperref}
\usepackage[nameinlink,noabbrev]{cleveref}
\usepackage{makeidx}

\hypersetup{
    colorlinks,
    citecolor=black,
    filecolor=black,
    linkcolor=black,
    urlcolor=black
}

\def\W{\mathcal W}
\def\R{\mathbb R}
\def\C{\mathbb C}
\def\N{\mathbb N}
\def\S{\mathbb{S}}
\def\T{\mathbb T}
\def\P{\mathbb P}
\def\E{\mathbb E}
\def\L{\mathrm L}
\def\H{\mathrm H}
\def\O{\Omega}
\def\eps{\varepsilon}
\def\D{\mathcal D}

\def \p {\partial}
\DeclareMathOperator{\supp}{supp}
\DeclareMathOperator{\re}{Re}
\DeclareMathOperator{\im}{Im}
\DeclareMathOperator{\vspan}{Span}
\def\d{\mathrm{d}}

\newtheorem{theo}{Theorem}[section]
\crefname{theo}{Theorem}{Theorems}
\Crefname{theo}{Theorem}{Theorems}

\newtheorem{prop}[theo]{Proposition}
\crefname{prop}{Proposition}{Propositions}
\Crefname{prop}{Proposition}{Propositions}

\newtheorem{lemme}[theo]{Lemma}
\crefname{lemme}{Lemma}{Lemmas}
\Crefname{lemme}{Lemma}{Lemmas}

\newtheorem{cor}[theo]{Corollary}
\crefname{cor}{Corollary}{Corollaries}
\Crefname{cor}{Corollary}{Corollaries}

\crefname{section}{Section}{Sections}
\Crefname{section}{Section}{Sections}

\crefname{equation}{}{}
\Crefname{equation}{}{}

\theoremstyle{definition}
\newtheorem{rem}[theo]{Remark}
\crefname{rem}{Remark}{Remarks}
\Crefname{rem}{Remark}{Remarks}

\newcommand{\e}{\mathrm{e}} 
\newcommand{\complexI}{\mathrm{i}} 

\DeclarePairedDelimiterX{\wick}[1]{:}{:}{#1} 

\DeclareMathOperator*{\argmin}{argmin}

\newcommand{\indic}[1]{\mathds{1}_{#1}}
\newcommand{\Atop}[2]{\genfrac{}{}{0pt}{1}{#1}{#2}} 

\numberwithin{equation}{section}

\bibliography{biblio.bib}

\title{Standing waves of the Anderson-Gross-Pitaevskii equation}

\author{S. Mackowiak\thanks{Univ Brest,  CNRS UMR 6205,  Laboratoire de Mathématiques de Bretagne Atlantique,  F-29200,  Brest,  France} \thanks{Université de Lorraine,  CNRS,  IECL,  F-54000 Nancy,  France}\\
\href{mailto:samael.mackowiak@univ-lorraine.fr}{samael.mackowiak@univ-lorraine.fr}}

\date{}

\begin{document}

\maketitle

\begin{abstract}
    In this paper, we study standing waves for the Anderson-Gross-Pitaevskii equation in dimension 1 and 2. The Anderson-Gross-Pitaevskii equation is a nonlinear Schrödinger equation with a confining potential and a multiplicative spatial white noise. Standing waves are characterized by a profile which is invariant by the dynamic and solves a nonlinear elliptic equation with spatial white noise potential. We construct such solutions via variational methods and obtain some results on their regularity, localization and stability.
\end{abstract}
\section{Introduction}

In this paper, we are interested in particular solutions of the Anderson-Gross-Pitaevskii
\begin{equation}{~}
    \begin{cases}\complexI\p_t u+ H u + \xi u+\lambda \left|u\right|^{2\gamma}u=0\\ u(0)=u_0,\end{cases} \label{Eq:AGP}
\end{equation} 
of the form $u(t,x)=\e^{i\omega t}\phi(x)$, where $H=\Delta-|x|^2$ is the Hermite operator, $\xi$ is a real spatial white noise, $\lambda\in\R$ and $\gamma>0$. Such a solution is called a standing wave of \cref{Eq:AGP} of frequency $\omega\in\R$ and profile $\phi$. The study of standing waves is motivated by the modeling of Bose-Einstein condensation. 
In fact, \cref{Eq:AGP} can be seen as a perturbation of the Gross-Pitaevskii equation
\begin{equation}{~}
    \begin{cases}\complexI\p_t u+ H u +\lambda \left|u\right|^{2\gamma}u=0\\ u(0)=u_0,\end{cases} \label{Eq-GP}
\end{equation} 
by a white noise potential. In the study of Bose-Einstein condensation, \cref{Eq-GP} appears as a mean-field limit of the many body Schrödinger equation, the nonlinear term appearing as a consequence of particle interactions \cite{BEC_Pitaevskii}. Adding a random potential in \cref{Eq-GP} can be seen as a way of modeling spatial inhomogeneities (see for example \cite{PhysRevLett.104.193901,PhysRevLett.109.243902} and references therein), thus one can see \cref{Eq-AGP} as a toy model of Bose-Einstein condensation in inhomogeneous environment with zero correlation length. In this context, at least formally, \cref{Eq:AGP} has two conserved quantities physically relevant, the mass 
$$\mathcal{M}(u)=\frac{1}{2}\int_{\R^d}|u|^2\d x$$
and the energy
$$\mathcal{E}(u)=\frac{1}{2}\langle-(H+\xi)u,u\rangle - \frac{\lambda}{2\gamma+2}\int_{\R^d}|u|^{2\gamma+2}\d x.$$
Standing waves of \cref{Eq:AGP} then appear naturally as solutions of minimal energy at fixed mass. Formally, if $\phi$ minimize the energy at fixed mass, by a Lagrange multiplier argument, there exists $\omega\in\R$ such that
$$-(H+\xi)\phi-\lambda|\phi|^{2\gamma}\phi+\omega \phi = 0$$
and $u(t,x)=\e^{i\omega t}\phi(x)$ is solution of \cref{Eq:AGP}.\\

In dimension 2, it is known from \cite{Mackowiak_2025,MackowiakStrichartzConfAnderson} that \cref{Eq:AGP} should be interpreted in a renormalized sense. This renormalization has to be taken into account in the study of standing waves. The construction of the renormalized operator has been conducted in \cite{MackowiakStrichartzConfAnderson} and will be recall in \cref{Sec-AndersonOp2d}. The renormalized operator is defined through an exponential transform inspired by the one introduced in \cite{Haire_Labbe} to study the continuous parabolic Anderson model, and later used in \cite{debussche_weber_T2,Tzvetkov_2023,tzvetkov2020dimensional} to solve the Anderson nonlinear Schrödinger equation on the torus $\T^2$ and in \cite{debussche2017solution,debussche2023global} to solve the same equation but on the full space $\R^2$. Let $Y=(-H)^{-1}\xi$, it is almost surely of regularity $1-$, so one can define $\rho=\e^Y$ and search for a formula for $(H+\xi)u$ when $u$ is of the form $u=\rho v$. A formal computation gives
$$(H+\xi)u=\rho\left(Hv+2\nabla Y\cdot \nabla v + xY\cdot xv + |\nabla Y|^2v\right).$$
Unfortunately, $\nabla Y$ is of regularity $0-$, thus the product $|\nabla Y|^2$ barely fails to exist. In \cite{Mackowiak_2025}, the author proved that the wick product $\wick{|\nabla Y|^2}$ can be defined as the almost sure limit, in suitable distribution spaces of regularity $0-$, of
$$|\nabla Y_N|^2-\E[|\nabla Y_N|^2],$$
where $Y_N$ is a suitable regularisation of $Y_N$.  Thus, the renormalized operator can be formally defined as
$$(H+\xi)u=\rho\left(Hv+2\nabla Y\cdot \nabla v + xY\cdot xv + \wick{|\nabla Y|^2}v\right).$$
The rigorous construction of $H+\xi$ rely on the associated quadratic form. 

Even in dimension 1, the construction of $H+\xi$ is not straightforward as $\xi$ is not a function. On compact interval with no potential, $\p_x^2+\xi$ has first been constructed in \cite{FukushimaNakao} using its quadratic form and has been studied in \cite{DumazLabbeAndersonLoc1D} using a twisted Sturm-Liouville approach, allowing to construct $\p_x^2+\xi$ on unbounded interval too. In \cref{Sec-AndersonOp1d}, we propose a construction of $H+\xi$ as a self-adjoint unbounded operator through its quadratic form. \cref{Prop-IPP1d} shows $\mathrm{D}(H+\xi)=\mathrm{D}(\p_x^2+\xi)\cap \mathrm{D}(|x|^2)$, where $\p_x^2+\xi$ is the unbounded self-adjoint operator on $\L^2(\R)$ constructed in \cite{DumazLabbeAndersonLoc1D} and $\mathrm{D}(|x|^2)=\{u\in\L^2(\R),xu\in\L^2(\R)\}$.\\

We are then interested in standing waves for
\begin{equation}{~}
    \begin{cases}\complexI\p_t u+ Au+\lambda \left|u\right|^{2\gamma}u=0\\ u(0)=u_0,\end{cases} \label{Eq-AGP}
\end{equation} 
in dimension $d\in\{1,2\}$, where $A=H+\xi$. As $A$ is independent of time, standing waves $u(t,x)=\e^{\complexI\omega t}\phi(x)$ solves \cref{Eq-AGP} if and only if $\phi$ is solution of
\begin{equation}
    -A\phi-\lambda|\phi|^{2\gamma}\phi+\omega \phi = 0. \label{Eq-StationaryAGP}
\end{equation}
This paper aim to construct standing waves in presence of a white noise potential and present some of their properties. 

When $\xi=0$ and $H=\Delta$, standing waves are well-known object. See for example \cite{NehariExistence,BerestyckiLions,StraussExistenceSolitary} for existence, \cite{BerestyckiLions} for non-existence, \cite{GidasSymmetry,KwongUniqueness} for uniqueness and \cite{BerestyckiInstabilité,CazenaveLionsStability,CazenaveStabilityInstability,WeinsteinSharpInterpolation,GrillakisStability,GrillakisStabilityII} for results on stability. Let us stress out many of these results relies on the particular structure of \cref{Eq-StationaryAGP} in this case, which is not preserved in presence of potentials. When $\xi=0$ and $H=\Delta-V$ with $V$ a regular and potentially unbounded potential, we refer the reader to \cite{KavianSelfSimilar,FukuizumiStabilityInstability,HiroseStructure,HiroseUniqueness,FukuizumiExponentialDecay,MaedaSymmetry,SelemBifurcation,GuoSeringerMassConcentration,PelinovskyGSEnergyCriticalGP} and references therein for many results on existence, uniqueness, localization and stability.

Nonlinear elliptic equations in presence of white noise potential on compact manifold have been studied in \cite{EulryVariational} using mountain pass methods and in \cite{ZhangConstraintMinAnderson} using constraint minimization. In both \cite{EulryVariational,ZhangConstraintMinAnderson}, the Anderson operator $\Delta+\xi$ is constructed using paracontrolled calculus. We will present here the construction of solution of \cref{Eq-StationaryAGP} through constraint minimization, namely energy minimization at fixed mass and action minimization. This allow to construct a positive solution of \cref{Eq-StationaryAGP} whenever such a solution exists.

\begin{theo}\label{Th-SynthesExistenceSW}
    Let $d\in\{1,2\}$, $\lambda,\omega\in\R$ and $\gamma>0$. Denote by $\mu_0$ the lowest eigenvalue of $-A$. In each of the following situation, there exists a positive solution of \cref{Eq-StationaryAGP} in $\mathrm{D}\left(\sqrt{|A|}\right)$,
    \begin{itemize}
        \item $\lambda>0$ and $\omega>-\mu_0$,
        \item $\lambda = 0$ and $\omega=\mu_0$,
        \item $\lambda<0$ and $\omega<-\mu_0$.
    \end{itemize}
    In every other cases, there is no non-trivial non-negative solution in $\mathrm{D}\left(\sqrt{|A|}\right)$.
\end{theo}

In order to do so, we use constrained minimization methods. In \cref{Sec-EnergyGS}, we study the minimization of energy at fixed mass and prove the following result.

\begin{theo}\label{Th-EnergyGS}

    Let $d\in\{1,2\}$, $\lambda,\omega\in\R$, $\gamma>0$ and $A$ be the self-adjoint operator constructed in \cref{Sec-AndersonOp1d,Sec-AndersonOp2d} with $\xi$ a spatial white noise, $\D^{1,2}=\mathrm{D}(\sqrt{|A|})$ and $\D^{-1,2}$ its dual. For $u\in\D^{1,2}$, we define the mass as
    $$\mathcal{M}(u)=\frac{1}{2}\int_{\R^d}|u|^2\d x$$
    and the energy as
    $$\mathcal{E}(u)=\frac{1}{2}\langle -Au,u\rangle_{\D^{-1,2},\D^{1,2}}-\frac{\lambda}{2\gamma+2}\int_{\R^d}|u|^{2\gamma+2}\d x.$$
    Then, it holds
    \begin{enumerate}
        \item For $\lambda\leqslant0$, for any $m>0$, there exists $\phi_m\in\D^{1,2}$ such that $\phi_m>0$ and
        $$\argmin_{\substack{u\in\D^{1,2}\\\mathcal{M}(u)=m}}\mathcal{E}(u) = \left\{\e^{\complexI\theta}\phi_m,\; \theta\in\R\right\}.$$
        \item For $\lambda>0$ and $\gamma<\frac{2}{d}$, for any $m>0$, there exists $\phi_m\in\displaystyle\argmin_{u\in\D^{1,2},\;\mathcal{M}(u)=m}\mathcal{E}(u)$ such that $\phi_m>0$.
        \item For $\lambda>0$ and $\gamma=\frac{2}{d}$, there exists $m_{*,\lambda}(\Xi)>0$ such that for all $m\in(0,m_{*,\lambda}(\Xi))$, there exists $\phi_m\in\displaystyle\argmin_{u\in\D^{1,2},\;\mathcal{M}(u)=m}\mathcal{E}(u)$ such that $\phi_m>0$. Moreover, there exists $m_{*,\lambda}(\Xi)\leqslant m^*_\lambda(\Xi)\leqslant m^*_\lambda$ such that
        $$\forall m>m^*_\lambda(\Xi), \inf_{\substack{u\in\D^{1,2}\\\mathcal{M}(u)=m}}\mathcal{E}(u)=-\infty,$$
        where $m^*_\lambda$ is the critical mass of the classical nonlinear Schrödinger equation (i.e.\ without any potential).
        \item For $\lambda>0$ and $\gamma>\frac{2}{d}$, for any $m>0$,
        $$\inf_{\substack{u\in\D^{1,2}\\\mathcal{M}(u)=m}}\mathcal{E}(u)=-\infty.$$
    \end{enumerate}
    
\end{theo}

It is classical that the set of minimal energy states of fixed mass is stable in $\D^{1,2}$ for \cref{Eq-AGP}. In particular, \cref{Th-EnergyGS} implies the orbital stability of the ground-state $\phi_m$ in the defocusing case ($\lambda<0$). \cref{Prop-ExistenceEnergyGS,Prop-FrequencyEnergyGS} show that, for $\lambda\leqslant 0$ or $\gamma<\frac{2}{d}$, energy minimization at fixed mass provide positive solutions of \cref{Eq-StationaryAGP} for any possible value of $\omega$. We also prove the small mass asymptotic of $\phi_m$ and its frequency $\omega_m$.\\

In \cref{Sec-ActionGS}, we investigate the minimization of action $\mathcal{S}_\omega = \mathcal{E} + \omega \mathcal{M}$. In the focusing case ($\lambda> 0$), one easily verifies the action is not lowerbounded. We thus introduce the Nehari manifold and show the following result.

\begin{theo}\label{Th-ACtionGS}

    Let $d\in\{1,2\}$, $\lambda,\omega\in\R$, $\gamma>0$ and $A$ the self-adjoint operator constructed in \cref{Sec-AndersonOp1d,Sec-AndersonOp2d} with $\xi$ a spatial white noise. Let $\D^{1,2}=\mathrm{D}(\sqrt{|A|})$ and $\D^{-1,2}$ its dual. Denote by $\mu_0$ the smallest eigenvalue of $-A$. Let $\omega>-\mu_0$ and define $$\mathcal{S}_\omega = \mathcal{E} + \omega \mathcal{M}$$ and $$\mathcal{N}_\omega =\left\{u\in\D^{1,2}\backslash\{0\},\; \langle \nabla\mathcal{S}_\omega u,u \rangle_{\D^{-1,2},\D^{1,2}}=0\right\}.$$
    Then, there exists $\psi_\omega\in\displaystyle\argmin_{\mathcal{N}_\omega}\mathcal{S}_\omega$ such that $\psi_\omega>0$.
    
\end{theo}

Let us stress out an important property of solution of \cref{Eq-StationaryAGP} compared to the noiseless case. In presence of white noise, standing waves have a limited Sobolev regularity.
\begin{theo}\label{Th-MaxRegSW}
    Let $d\in\{1,2\}$, $\lambda,\omega\in\R$ and $\gamma>0$. Almost surely any solution $u$ of \cref{Eq-StationaryAGP} belongs to $\W^{2-\frac{d}{2}-,q}$ for all $q\in[2,+\infty)$, where $\W^{s,q}$ spaces are defined in \cref{Sec-ConfiningSobolev}. Moreover, almost surely any solution $u$ of \cref{Eq-StationaryAGP} which belongs to $\mathrm{H}^{2-\frac{d}{2},q}(U)$ for some open set $U\subset\R^d$ and $q\in[2,+\infty]$ vanishes everywhere in $U$.
\end{theo}

\section{Preliminaries}\label{Sec-Prelim}

\subsection{Notations and conventions}

\begin{itemize}
    \item $\N=\{0,1,2,\dotsc\}$ and $\N^*=\{1,2,\dotsc\}$
    \item $d\in\N$ is the space dimension
    \item $\Delta = \sum_{j=1}^d \p_{x_j}^2$ is the laplacian
    \item For $x\in\mathbb{R}^d$, $|x|^2=\sum_{j=1}^d x_j^2$ and $\langle x\rangle^2 = 1+|x|^2$
    \item For $\alpha\in\N^d$, $|\alpha|=\sum_{j=1}^d\alpha_j$ and $\p^\alpha = \prod_{j=1}^d\p_{x_j}^{\alpha_j}$. We write $\alpha\leqslant\beta$ of $\alpha_j\leqslant \beta_j$ for all $1\leqslant j\leqslant d$.
    \item Let $E\subset\R^d$, then $\indic{E}:\R^d\to\{0,1\}$ denotes the characteristic function of $E$.
    \item For $p\in[1,+\infty)$ and a borelian $U\subset\R^d$, $\L^p(U)$ denotes the set of measurable functions $f:U\to\C$ such that $|f|^p$ is Lebesgue-integrable on $U$, modulo almost everywhere equality. Likewise, $\L^\infty(U)$ denotes the set of function bounded almost everywhere on $U$, modulo almost everywhere equality. When $p=2$, we endow $\L^2(U)$ with the inner product
    $$\forall f,g\in\L^2(U),\; (f,g)_{\L^2}=\re\int_U f(x) \overline{g(x)}\d x.$$ 
    Moreover, for $p\in[1,+\infty]$, $\L^p_{loc}(U)$ denotes the set of measurable functions $f$ on $U$, such that for any compact set $K\subset U$, $f\indic{K}\in\L^p(U)$.
    \item Let $E$ be a Banach space, $p\in[1,+\infty]$ and $U\subset\R^d$ be a borelian. Then $\L^p(U,E)$ denotes the set of measurable functions $f:U\to E$ such that $|f|_E\in\L^p(U)$.
    \item We denote by $\mathcal{C}^\infty_c(\R^d)$ the space of smooth and compactly supported functions on $\R^d$ and by $\D'(\R^d)$ its dual, that is the space of distributions on $\R^d$, endowed with the real bracket,  
    $$\forall T\in\L^1_{loc}(\mathbb{R}^d),\forall \phi\in\mathcal{C}^\infty_c(\mathbb{R}^d), \langle T,\phi\rangle_{\D'} = \re\left(\int_{\mathbb{R}^d} T(x) \overline{\phi(x)}\d x\right).$$
    \item We denote by $\mathcal{S}(\R^d)$ the Schwartz space on $\R^d$ and $\mathcal{S}'(\R^d)$ its dual, that is the space of tempered distributions on $\R^d$.
    \item For $s\in\R$ and $p\in(1,+\infty)$, the Sobolev spaces are defined as
    \[
    \mathrm{H}^{s,p}=\left\{u \in\mathcal{S}'(\mathbb{R}^d), \mathcal{F}^{-1}\left(\langle \eta \rangle^s \mathcal{F}u\right)\in\L^p(\mathbb{R}^d)\right\}
    \]
    endowed with the norm
    \[
    |u|_{\mathrm{H}^{s,p}_x} =  |\mathcal{F}^{-1}\left(\langle \eta \rangle^s \mathcal{F}u\right)|_{\L^p_x},
    \]
    where $\mathcal{F}$ is the Fourier transform. For $p\in\{1,+\infty\}$ and $k\in\N$, they are defined as
    \[
    \mathrm{H}^{k,p}=\left\{u\in\L^p(\R^d),\; \forall \alpha\in\N^d,\; |\alpha|\leqslant k \Rightarrow \p^\alpha u\in\L^p(\R^d)\right\}
    \]
    and extended by complex interpolation for $s\geqslant 0$.
    We also use the notation $\mathrm{H}^s=\mathrm{H}^{s,2}$.
    \item Let $E$ be a Banach space, we write $|\cdot|_E$ its norm, $\mathcal{L}(E)$ the set of continuous endomorphisms on $E$, and for $F$ another Banach space, $\mathcal{L}(E,F)$ the set of bounded linear maps from $E$ to $F$. In particular, the dual of $E$, defined by $\mathcal L(E,\R)$, is denoted $E'=\mathcal{L}(E,\R)$, and the duality bracket is defined as follows:
    \[
    \forall e\in E,\; \forall f\in E',\; \langle f,e\rangle_{E',E}=f(e).
    \]
    \item We write $A\lesssim B$ if there exists a constant $C>0$ independent of $A$ and $B$ such that $A\leqslant C B$. We write $A\approx B$ if $A\lesssim B$ and $B\lesssim A$. If $C$ depends on a parameter $\theta$, we may write $A\lesssim_\theta B$.
\end{itemize}

\subsection{Hermite-Sobolev spaces}\label{Sec-ConfiningSobolev}

We denote by $H=\Delta-|x|^2$ the Hermite operator. It is well-known that this operator has eigenfunctions $(h_k)_{k\in\N}$ verifying the relation $-H h_k = \lambda_k^2 h_k$ where $\lambda_k^2\sim ck^{\frac{1}{d}}$. Moreover, $(h_k)_{k\in\N}$ is a complete orthonormal system of $\L^2(\mathbb{R}^2,\R)$.\\

Let $s\in\R$ and $p\in(1,+\infty)$ and define
$$\W^{s,p} =\left\{u\in\S'(\mathbb{R}^2), (-H)^{\frac{s}{2}}u\in\L^p(\mathbb{R}^2)\right\}$$
endowed with the norm $$|u|_{\W^{s,p}_x} = |(-H)^{\frac{s}{2}}u|_{\L^p_x}.$$
It is known that these spaces are Banach spaces and that for $q$ such that $\frac{1}{p}+\frac{1}{q}=1$, we have $\left(\W^{s,p}\right)'=\W^{-s,q}$ with equal norm. Moreover, the subspace $\vspan\left\{h_k, k\in\N\right\}$ is dense in $\W^{s,q}$-spaces (see for example \cite{BongioanniHSspaces} for the case $s\geqslant0$) and for $s\geqslant0$, an equivalent norm is given for $p\in(1,+\infty)$ by (see \cite{Dziubanski})
$$|u|_{\W^{s,p}}\approx|u|_{\H^{s,p}_x}+|\langle x\rangle^s u|_{\L^p_x}.$$
This shows $\W^{s,p}=\left\{u\in \H^{s,p}(\mathbb{R}^2), \langle x\rangle^s u\in\L^p(\mathbb{R}^2)\right\}.$ Using the norm equivalence, we define $\W^{2,p}$ with $p\in\{1,+\infty\}$ as
$$\W^{2,p} = \left\{u\in \H^{2,p}(\mathbb{R}^2), \langle x\rangle^{2} u\in\L^p(\mathbb{R}^2)\right\}$$
endowed with the norm
$$|u|_{\W^{2,p}_x} = |u|_{\H^{2,p}_x}+|\langle x\rangle^{2} u|_{\L^p_{x}}$$
and extend this definition to $s\in[0,2]$ by complex interpolation. Moreover inspired by the relation $\left(\W^{s,p}\right)'=\W^{-s,q}$, for $s\in[-2,0)$, we define $\W^{s,\infty} = \left(\W^{-s,1}\right)'$. As usual, we cannot define spaces of negative $\L^1$ regularity as dual spaces of positive $\L^\infty$ regularity.

\begin{rem}{~}
    We do not claim that the first definition of $\W^{s,p}$ agrees with the one we give for spaces over $\L^1$ and $\L^\infty$. We use the same notation for simplicity as we will always control $\W^{s,\infty}$ norms by Sobolev embeddings.
\end{rem}

The choices we made ensure us that we can interpolate between spaces of positive regularity and some other convenient properties that we need in our analysis. For example, our choices allow us to have a simple product rule on our spaces, whose proof follows from Leibniz rule, Hölder inequality and complex bilinear interpolation. We start with product rules, recalling an important estimate going back to Kato and Ponce \cite{KatoPonce}. 
\begin{prop}{(Theorem 1.4 in \cite{GulisashviliKon96})} \label[prop]{Prop-Kato-Ponce}
    Let $r\in(1,+\infty)$ and $s\geqslant 0$. For any $1<p_1,q_1,p_2,q_2\leqslant +\infty$ with $\frac{1}{r}=\frac{1}{p_j}+\frac{1}{q_j}$ ($j\in\{1,2\}$), there exists $C>0$ such that, for any $f\in \mathrm{H}^{s,p_1}\cap\L^{p_2}$ and any $g\in \mathrm{H}^{s,q_2}\cap\L^{q_1}$, it holds
    $$|fg|_{\mathrm{H}^{s,r}}\leqslant C\left(|f|_{\mathrm{H}^{s,p_1}}|g|_{\L^{q_1}}+|f|_{\L^{p_2}}|g|_{\mathrm{H}^{s,q_2}}\right).$$    
\end{prop}

The proof of \cref{Lem-ProductRule,Cor-ProductRuleNegPosReg,Prop-action_d/dx,Sobolev-embeddings} below are given in \cite{Mackowiak_2025}.

\begin{lemme}\label[lemme]{Lem-ProductRule}
    Let $s\geqslant0$ and $1\leqslant p,q,r\leqslant+\infty$ such that $\frac{1}{p}+\frac{1}{q}=\frac{1}{r}$. There exists a constant $C>0$ such that for all $u\in\W^{s,p}$ and $v\in \mathrm{H}^{s,q}$, we have $uv\in\W^{s,r}$ and $|uv|_{\W^{s,r}}\leqslant C |u|_{\W^{s,p}}|v|_{\mathrm{H}^{s,q}}$.   
\end{lemme}

\begin{cor}\label[cor]{Cor-ProductRuleNegPosReg}
    Let $s\geqslant0$, $1<p<+\infty$ and $1\leqslant q,r\leqslant+\infty$ such that $1-\frac{1}{p}+\frac{1}{q}=1-\frac{1}{r}$. There exists a constant $C>0$ such that for all $u\in\W^{-s,r}$ and $v\in \mathrm{H}^{s,q}$, we have $uv\in\W^{-s,p}$ and $|uv|_{\W^{-s,p}}\leqslant C |u|_{\W^{-s,r}}|v|_{\mathrm{H}^{s,q}}$.
\end{cor}

On $\W^{s,q}$ spaces, derivation and multiplication by a power of $\langle x\rangle$ act directly on the regularity exponent.

\begin{prop}(see Corollary 2.5 in \cite{Mackowiak_2025})  \label[prop]{Prop-action_d/dx}
    Let $s\in\R$, $j\in\{1,2\}$ and $q\in(1,+\infty)$, then $\p_j,x_j\in\mathcal{L}(\W^{s,q},\W^{s-1,q})$.
    
\end{prop}   

The following corollary follows from \cref{Prop-action_d/dx} by Stein's complex interpolation method \cite{SteinInterpolation,VoigtAbstractInterpolation}.

\begin{cor}\label[cor]{Cor-action_V}

    Let $\alpha\geqslant 0$, $s\in\R$ and $q\in(1,+\infty)$. Then, $\langle x\rangle^{\alpha}\in\mathcal{L}(\W^{s,q},\W^{s-\alpha,q})$.
    
\end{cor}

The Sobolev-Hermite spaces verify some continuous embeddings as in the classical Sobolev framework. We will also refer to these embeddings as Sobolev embeddings.

\begin{prop}{(Sobolev embeddings)}\label[prop]{Sobolev-embeddings}
    Let $1< p\leqslant q<+\infty$ and $s>\sigma$ such that $\frac{1}{p}-\frac{s}{2}\leqslant\frac{1}{q}-\frac{\sigma}{2}$. Then $\W^{s,p}$ is continuously embedded in $\W^{\sigma, q}$. Moreover if $1<p< q\leqslant+\infty$ and $s>\sigma$ such that $\frac{1}{p}-\frac{s}{2}<\frac{1}{q}-\frac{\sigma}{2}$ and $\sigma\in[-2,2]$ if $q=+\infty$, then $\W^{s,p}$ is compactly embedded in $\W^{\sigma,q}$.
    
\end{prop}

We can also gain integrability for the elements of $\W^{s,2}$, using the trade-off between localization and regularity in confining Sobolev spaces.
\begin{prop}\label[prop]{Reverse-Sobolev-embeddings}
    Let $p\in[1,2]$ and $s,\sigma\geqslant 0$ such that $s>\sigma+d\left(\frac{1}{p}-\frac{1}{2}\right)$. Then $\W^{s,2}$ is compactly embedded in $\W^{\sigma,p}$.
\end{prop}
\begin{proof}
    The compactness of the embedding follows from the sharpness of the condition $s>\sigma+d\left(\frac{1}{p}-\frac{1}{2}\right)$ and the compact embedding $\W^{s,2}\subset\W^{s-\eps,2}$ for any $\eps>0$. Assume $\sigma=0$ first and let $u\in\W^{s,2}$. Then, by Hölder's inequality,
    $$|u|_{\L^p}\leqslant \left|\langle x\rangle^{-s}\right|_{\L^r}|u|_{\W^{s,2}}$$
    for $\frac{1}{r}=\frac{1}{p}-\frac{1}{2}$. As $s>d\left(\frac{1}{p}-\frac{1}{2}\right)$, $sr>d$ and thus $\left|\langle x\rangle^{-s}\right|_{\L^r}<+\infty$. This proves the case $\sigma=0$. The case $\sigma\in2\N$ follows from \cref{Prop-action_d/dx,Cor-action_V}. Finally, the general case $\sigma\geqslant0$ follows by complex interpolation. 
\end{proof}

\subsection{White noise}

A real centered gaussian random field $\{\langle\xi,\phi\rangle,\; \phi\in\mathcal{S}(\R^d)\}$ defined on $(\Omega,\mathcal{A},\P)$ is called a white noise if it verifies the Wiener isometry
$$\forall \phi,\psi\in\mathcal{S}(\R^d),\; \E\left[\langle\xi,\phi\rangle\langle\xi,\psi\rangle\right]=(\phi,\psi)_{\L^2},$$
allowing to extend it as an isometry $\xi:\L^2(\R^d)\to\mathbb{L}^2(\Omega)$. For any orthonormal basis $(e_n)_{n\in\N}$ of $\L^2(\R^d)$, one can show
$$\xi = \sum_{n\in\N}\langle\xi,e_n\rangle e_n,$$
where the series converges almost surely in any Hilbert space $\L^2(\R^d)\subset\mathcal{H}\subset\D'(\R^d)$ such that the canonical embedding of $\L^2(\R^d)$ in $\mathcal{H}$ is Hilbert-Schmidt (see for example Proposition 13.5 in \cite{gammaRadonSurvey}). Recall that a bounded linear map $T:\mathcal{H}_0\to\mathcal{H}_1$ is Hilbert-Schmidt if there exists an orthonormal basis $(e_n)_{n\in\N}$ of $\mathcal{H}_0$ verifying
$$\sum_{n\in\N}|Te_n|_{\mathcal{H}_1}^2 <+\infty.$$
If this sum is finite for some orthonormal basis, it is then finite for every orthonormal basis, with the same value. This shows, almost surely, $\xi$ is defined as a random Schwartz distribution. We now investigate the regularity of $\xi$ in confining Sobolev spaces. We introduce the kernel of $(-H)^{-1}$,
\begin{equation}\label{Eq-DefK}
    \forall x,y\in\R^d,\; K(x,y) = \sum_{n\in\N} \frac{h_n(x)h_n(y)}{\lambda_n^2}.
\end{equation}

The following lemma extends Proposition 4 in \cite{debouard_2d_GP} which deals with the case $d=2$.

\begin{lemme}\label[lemme]{Lem-KLrWa2}
    Let $d\in\N^*$ and $q\in[2,+\infty]$. For any $\alpha<2-\frac{d}{2}-\frac{d}{q}$, $K\in\L^q_x\W^{\alpha,2}_y$.
\end{lemme}

\begin{proof}
    We only prove the cases $q=2$ or $q=+\infty$, as the general result follows by complex interpolation. First, assume $q=+\infty$. Let $\sigma > \frac{d}{2}$, then Sobolev embeddings implies
    $$\forall f\in\mathcal{S}(\R^d),\; \sup_{x\in\R^d}\left|\int_{\R^d}K(x,y)f(y)\d y\right| = \left|(-H)^{-1}f\right|_{\L^\infty_x}\lesssim \left|(-H)^{-1}f\right|_{\W^{\sigma,2}_x}\lesssim \left|f\right|_{\W^{\sigma-2,2}_x}.$$
    By density of $\mathcal{S}(\R^d)$ in $\W^{\sigma-2,2}$, it holds $K\in\L^\infty_x\W^{2-\sigma,2}_y$. Now, assume $q=2$. Then, using the Kernel theorem, 
    $$|K|_{\L^2_x\W^{\alpha,2}_y}^2 = \sum_{k\in\N} \left|(-H)^{\frac{\alpha}{2}-1}h_k\right|_{\L^2}^2=\sum_{k\in\N} \lambda_k^{2\alpha-4}.$$
    Thus, $K\in\L^2_x\W^{\alpha,2}_y$ if and only if $\alpha<2-d$, as $\lambda_k^2\sim k^{\frac{1}{d}}$.
    
\end{proof}

\begin{cor}\label[cor]{Cor-RegWN}
    Let $d\in\N^*$. Almost surely, for any $q\in[2,+\infty]$, $s>\frac{d}{q}>\sigma$, $\xi\in\W^{-\frac{d}{2}-s,q}$, and $\xi\notin\W^{-\frac{d}{2}-\sigma,q}$.
\end{cor}
\begin{proof}
    By Sazonov's theorem, almost surely
    $$\xi\in\W^{-\frac{d}{2}-s,2} \iff s>\frac{d}{2}$$
    as the injection of $\L^2(\R^d)$ in $\W^{-\frac{d}{2}-s,2}$ is Hilbert-Schmidt if and only if $s>\frac{d}{2}$. Moreover, by \cref{Reverse-Sobolev-embeddings}, for any $q\in[2,+\infty]$ and any $\sigma<\frac{d}{q}$, $\W^{-\frac{d}{2}-\sigma,q}\subset\W^{-d,2}$. Thus, almost surely, for any $q\in[2,+\infty]$ and any $\sigma<\frac{d}{q}$, $\xi\notin\W^{-\frac{d}{2}-\sigma,q}$. In particular, almost surely $\xi\in\mathcal{S}'(\R^d)$. Let $N,M\in\N$ and $$\xi_N = \sum_{n=0}^N \xi_n h_n.$$ Almost surely, $\xi_N$ converges to $\xi$ in $\mathcal{S}'(\R^d)$. Assume $N>M$. Let $q\in(2,+\infty)$ and $\alpha\in\R$, then by Fubini and gaussianity, it holds
    \begin{align*}
        \E\left[|\xi_N-\xi_M|_{\W^{\alpha,q}}^q\right] &= \int_{\R^d} \E\left[\left|(-H)^{\frac{\alpha}{2}}(\xi_N-\xi_M)(x)\right|^q\right]\d x\\
        &\lesssim_q \int_{\R^d} \E\left[\left|(-H)^{\frac{\alpha}{2}}(\xi_N-\xi_M)(x)\right|^2\right]^{\frac{q}{2}}\d x &\lesssim_q \int_{\R^d} \left|\sum_{n=M+1}^N \lambda_n^{2\alpha} |h_n(x)|^2\right|^{\frac{q}{2}}\d x.
    \end{align*}
    Now, remark
    $$|K(x)|_{\W^{\sigma,2}}^2 = \sum_{n\in\N}\lambda_n^{2\sigma-4}|h_n(x)|^2,$$
    thus for $\eps>0$ small enough,
    $$\E\left[|\xi_N-\xi_M|_{\W^{\alpha,q}}^q\right] \lesssim_q \lambda_M^{-2\eps} |K|_{\L^q_x\W^{2+\alpha+\eps,2}_y}^q.$$
    For $\alpha=-\frac{d}{2}-s$, one can choose $\eps>0$ such that $\alpha+\eps<-\frac{d}{2}-\frac{d}{q}$ and \cref{Lem-KLrWa2} ensures $|K|_{\L^q_x\W^{2+\alpha+\eps,2}_y}^q<+\infty$. Hence, $\xi_N $ converges to $\xi$ in $\mathbb{L}^q(\Omega,\W^{-\frac{d}{2}-s,q})$ and $\xi\in\W^{-\frac{d}{2}-s,q}$ a.s. The case $q=+\infty$ follows from Sobolev embeddings.
\end{proof}

\subsection{Confining Anderson operator in dimension 1}\label{Sec-AndersonOp1d}

A first approach to define $A = H + \xi$ could be to follow the one used in \cite{DumazLabbeAndersonLoc1D,MinamiAiry}. Recall that, in dimension 1, $\xi=B'$ where $B$ is a brownian motion on $\R$. Thus, for $u$ "sufficiently regular" a formal computation gives
$$Au = u''+B'u-x^2u = u''+(Bu)'-Bu'-x^2u=(u'+Bu)'-Bu'-x^2u.$$
This formal computation allows to define $A$, in a distributional sense, on the set of absolutely continuous functions as
$$A:u\in\mathcal{AC}(\R)\mapsto (u'+Bu)'-Bu'-x^2u\in\D'(\R).$$
In order to define it as an unbounded operator over $\L^2(\R)$, one can define it on $$ \mathrm{D}_0=\left\{u\in \L^2(\R)\cap\mathcal{AC}(\R),\; u'+Bu\in\mathcal{AC}(\R) \text{ and } (u'+Bu)'-Bu',x^2u\in\L^2(\R)\right\},$$
but showing $(A,\mathrm{D}_0)$ defines a self-adjoint operator is then a bit difficult. In order to circumvent the difficulty of the Sturm-Liouville approach, one can try to define $A=H+\xi$ through its quadratic form. This idea goes back to \cite{FukushimaNakao} where they constructed the quadratic form of $\p_x^2+\xi$ with Dirichlet conditions on bounded intervals. The construction we present extends to any distributional potential $\xi\in\W^{-\kappa,\infty}$ with $\kappa\in(0,1)$.\\

A formal integration by part gives
$$(-(H+\xi)u,v)=(-Hu,v)-(\xi u,v)=\re\int_\R u'\overline{v'}+x^2u\overline{v}\d x -(\xi,\overline{u}v).$$
Define the symmetric bilinear form
\begin{equation}\label{Eq-QuadFormA1d}
     \forall u,v\in\W^{1,2},\; a(u,v) = \re\int_\R u'\overline{v'}+x^2u\overline{v}d x - \langle \xi, \overline{u}v\rangle_{\W^{-1,\infty},\W^{1,1}}.
\end{equation}
Using Cauchy-Schwartz inequality and \cref{Lem-ProductRule}, it holds
$$\forall u,v\in\W^{1,2},\; |a(u,v)|\lesssim (1+|\xi|_{\W^{-1,\infty}}) |u|_{\W^{1,2}}|v|_{\W^{1,2}}.$$
It shows the symmetric bilinear form $a$ is bounded on $\W^{1,2}$. In order to apply Theorem 6.2.6 in \cite{KatoPerturbationBook} and construct $-A$, it would be sufficient to prove that $a$ is coercive. Unfortunately, it is not in general but it is quasi-coercive.

\begin{prop}\label[prop]{Prop-Quasi-coercivity-a-1d}
    There exists $\delta=\delta(\xi)>0$ such that
    $$\forall u\in\W^{1,2},\; a(u,u)+\delta|u|_{\L^2}^2\geqslant \frac{1}{2}|u|_{\W^{1,2}}^2.$$
\end{prop}

\begin{proof}

    Remark for any $u,v\in\W^{1,2}$, as $\xi\in\W^{-\kappa,\infty}$ with $\kappa\in(0,1)$, 
    $$\langle \xi, |u|^2\rangle_{\W^{-1,\infty},\W^{1,1}} = \langle \xi u,u\rangle_{\W^{-\kappa,2},\W^{\kappa,2}}$$
    and duality and \cref{Cor-ProductRuleNegPosReg} imply
    $$|\langle \xi u,u\rangle_{\W^{-\kappa,2},\W^{\kappa,2}}|\lesssim |\xi|_{\W^{-\kappa,\infty}}|u|_{\W^{\kappa,2}}^2.$$
    By interpolation and Young inequality, there exists $\delta>0$ such that
    $$|\langle \xi u,u\rangle_{\W^{-\kappa,2},\W^{\kappa,2}}|\leqslant \frac{1}{2}|u|_{\W^{1,2}}^2+\delta|u|_{\L^2}^2,$$
    thus
    $$a(u,u)+\delta|u|_{\L^2}^2\geqslant \frac{1}{2}|u|_{\W^{1,2}}^2.$$
    
\end{proof}

We can then apply Theorem 6.2.6 in \cite{KatoPerturbationBook} to construct $-A+\delta$ as a positive self-adjoint operator on 
$$\mathrm{D}(A)=\left\{u\in\W^{1,2},\; \exists C>0,\; \forall \phi\in\W^{1,2},\; |a(u,\phi)|\leqslant C|\phi|_{\L^2}\right\}.$$
As $\W^{1,2}$ is compactly embedded in $\L^2(\R)$, $A$ has a discrete spectrum.

Using the exponential transform introduced in \cite{Haire_Labbe} and extensively used in the 2d case \cite{debussche_weber_T2,debussche2017solution,debussche2023global,tzvetkov2020dimensional,Tzvetkov_2023,Mackowiak_2025,mouzard2023simple}, we can obtain an exact definition of the domain. Let $Y=(-H)^{-1}\xi$, then $Y\in\H^{1,\infty}$. Then one can define $\rho=\e^Y\in \H^{1,\infty}$ and for any $u=\rho v$, it holds formally
\begin{equation}
    Au = \rho (Hv+2Y'v'+xY\cdot xv+|Y'|^2v).\label{Eq-FormulaA1d}
\end{equation}
This formula is true for any $u=\rho v\in \W^{1,2}$ (recall \cref{Lem-ProductRule} implies $\rho\W^{1,2}=\W^{1,2}$) and the we deduce an explicit definition of the domain.

\begin{prop}\label[prop]{Prop-LocDomain1d}

    Let $Y=(-H)^{-1}\xi$ and $\rho=\e^Y$. Then $\mathrm{D}(A)=\rho\W^{2,2}$. In particular, for any $u\in\mathrm{D}(A)$, $x^2u\in\L^2(\R)$.
    
\end{prop}

\begin{proof}[Proof of \cref{Prop-LocDomain1d}.]

    It is easy to check $\rho\W^{2,2}\subset \mathrm{D}(A)$ from \cref{Eq-FormulaA1d}. We now show the converse. As $\mathrm{D}(A)\subset\W^{1,2}$, for any $u\in \mathrm{D}(A)$, $v=u\rho^{-1}\in\W^{1,2}$ and there exists $f\in\L^2(\R)$ such that $Au=f$. Using \cref{Eq-FormulaA1d}, it is equivalent to
    $$Hv = \rho^{-1}f - 2Y'v' - x^2Yv - |Y'|^2 v.$$
    Remark $Y',xY\in\L^\infty(\R)$ and $v,v',xv\in\L^2(\R)$, thus $Hv\in\L^2(\R)$ i.e. $v\in\W^{2,2}$.
    
\end{proof}

Inspired by this result, we define
 \begin{equation}
     \forall s\in[0,2),\; \forall p\in(1,+\infty),\;  \D^{s,p}=\rho\W^{s,p} \text{ and } \D^{-s,p}=(\D^{s,p'})',\label{Def-Dsq}
 \end{equation}
with $\frac{1}{p}+\frac{1}{p'}=1$.\\

Remark that in dimension 1, $\W^{1,1}\subset\H^{1,1}(\R)\subset\L^\infty(\R)$, thus for any $p\in(1,+\infty]$, $\L^p(\R)\subset\W^{-\frac{1}{p},\infty}$ and the previous construction applies to $A+W$ with $W\in\L^p$. \cref{Prop-PerturbationOp1d} below shows the domain of $A$ is stable under the perturbation by a potential in $\L^2(\R,\R)+\L^\infty(\R,\R)$. Recall that $\L^p(\R,\R)\subset \L^2(\R,\R)+\L^\infty(\R,\R)$ for any $p\in[2,+\infty]$, thus, \cref{Prop-PerturbationOp1d} provide the stability of $\mathrm{D}(A)$ under a quite large class of perturbations.
 \begin{prop}\label[prop]{Prop-PerturbationOp1d}
     Let $W\in\L^2(\R,\R)+\L^\infty(\R,\R)$. Then $(A+W,\mathrm{D}(A))$ is a self-adjoint operator.
 \end{prop}
 \begin{proof}
    By construction $\mathrm{D}(A)\subset\W^{1,2}\subset\L^2(\R)\cap\L^\infty(\R)$. Thus, for any $W\in\L^2(\R,\R)+\L^\infty(\R,\R)$, $A+W:\mathrm{D}(A)\to\L^2(\R)$ is well-defined and symmetric as $W$ is real-valued. As $W\in\W^{-\frac{1}{2}-,\infty}$, there exists $\mathrm{D}(A+W)\subset\W^{1,2}$ such that $(A+W,\mathrm{D}(A+W))$ is a self-adjoint operator. Thus, $A+W:\mathrm{D}(A)\to\L^2(\R)$ implies $\mathrm{D}(A)\subset\mathrm{D}(A+W)$. Likewise, using $A=(A+W)-W$ and  $\mathrm{D}(A+W)\subset\L^2(\R)\cap\L^\infty(\R)$, one obtain $A:\mathrm{D}(A+W)\to\L^2(\R)$, thus $\mathrm{D}(A)=\mathrm{D}(A+W)$.  
 \end{proof}

Moreover, one can show simplicity of the spectrum using an elementary ODE argument.

\begin{prop}\label[prop]{Prop-SpectralGapA1d}
    Let $W\in\L^2(\R,\R)+\L^\infty(\R,\R)$. The spectrum of $-A+W$ is simple. Moreover, $-A+W$ the eigenfunction associated to the lowest eigenvalue can be chosen positive.
\end{prop}

Using Sazonov theorem, one can prove a maximal local Sobolev regularity for elements of $\mathrm{D}(A)$.

 \begin{cor}\label[cor]{Cor-MaximalRegDomain1d}
 
     Almost surely, any $u\in\mathrm{D}(A)$ which belongs to $\mathrm{H}^{\frac{3}{2},q}(I)$ for some interval $I\subset\R$ and $q\in[2,+\infty]$ vanishes everywhere on $I$.
     
 \end{cor}

 \begin{proof}

     Recall $Y=(-H)^{-1}\xi$. We work on the event 
     $$\O_0=\{Y\in\mathrm{H}^{1,\infty}(\R)\}\cap\bigcap_{r\in\mathbb{Q}}\bigcap_{n\in\N}\left\{\xi\notin\mathrm{H}^{-\frac{1}{2}}\left(\left[r-2^{-n},r+2^{-n}\right]\right)\right\}$$
     which is of probability 1 by \cref{Cor-RegWN}, Sazonov theorem and Sobolev embeddings. By \cref{Prop-LocDomain1d,Cor-RegWN,Cor-action_V,Prop-action_d/dx,Sobolev-embeddings}, $\mathrm{D}(A)=\rho\W^{2,2}$ with $\rho=\e^Y$. Let $v\in\W^{2,2}$ such that $v_{|I}\neq 0$ and $u=\rho v\in \mathrm{H}^{\frac{3}{2},q}(I)$. Without lost of generality, one can assume $I$ is bounded and thus $u=\rho v\in\mathrm{H}^{\frac{3}{2}}(I)$. Then, $u'=\rho (v'+Y'v)\in\H^{\frac{1}{2}}(I)$. As $\rho\in\H^{1,\infty}(I)$ and $v'\in\H^1(I)$, it implies $Y'v\in \mathrm{H}^{\frac{1}{2}}(I)$. As $v_{|I}\neq 0$ and $\W^{2,2}\subset\mathcal{C}^1(\R)$, there exists a compact interval $J\subset I$ and $c>0$ such that 
     $$\forall x\in J,\; |v(x)|>c.$$ Thus $v^{-1}\in\L^\infty(J)$ and $\left(v^{-1}\right)'=-v^{-2}v'\in\L^\infty(J)$. Hence $v^{-1}\in \mathrm{H}^{1,\infty}(J)$ and by Kato-Ponce inequality, $Y'\in \mathrm{H}^{\frac{1}{2}}(J)$. As $x^2Y\in\L^\infty(J)\subset\L^2(J)$, it holds $\xi\in\mathrm{H}^{-\frac{1}{2}}$, which is impossible on $\O_0$. Thus $v_{|I}= 0$, i.e.\ $u_{|I}= 0$.
     
 \end{proof}

As already mentioned, we could have defined $A = H + \xi$ as in \cite{MinamiAiry,DumazLabbeAndersonLoc1D}, using $\xi=B'$ where $B$ is a brownian motion on $\R$ and the formal Sturm-Liouville formula.
$$Au = u''+B'u-x^2u = u''+(Bu)'-Bu'-x^2u=(u'+Bu)'-Bu'-x^2u.$$
This formal computation allows to define $A$, in a distributional sense, on the set of absolutely continuous functions as
$$A:u\in\mathcal{AC}(\R)\mapsto (u'+Bu)'-Bu'-x^2u\in\D'(\R).$$
In order to define it as an unbounded operator over $\L^2(\R)$, one can define it on $$ \mathrm{D}_0=\left\{u\in \L^2(\R)\cap\mathcal{AC}(\R),\; u'+Bu\in\mathcal{AC}(\R) \text{ and } (u'+Bu)'-Bu',x^2u\in\L^2(\R)\right\},$$
but showing $(A,\mathrm{D}_0)$ defines a self-adjoint operator is then a bit difficult. However, when $A$ is already defined as a self-adjoint operator, verifying that $\mathrm{D}(A)=\mathrm{D}_0$ is easier. Indeed, for $u\in \mathrm{D}(A)$, \cref{Prop-LocDomain1d} implies $u$ is sufficiently regular so that $$Au = Hu+\xi u = (u'+Bu)'-Bu'-x^2u$$ holds in $\D'(\R)$.

When there is no noise, it is known that $\mathrm{D}(-\p_x^2+x^2)=\mathrm{D}(-\p_x^2)\cap \mathrm{D}(x^2)$. The following result shows $\mathrm{D}(A)=\mathrm{D}(\p_x^2+\xi)\cap \mathrm{D}(x^2)$, where $\p_x^2+\xi$ is the self-adjoint operator constructed in \cite{DumazLabbeAndersonLoc1D}.

\begin{prop}\label[prop]{Prop-IPP1d}
    The domain of $A$ is given by
    $$\mathrm{D}(A) =\left\{u\in \L^2(\R)\cap\mathcal{AC}(\R),\; u'+Bu\in\mathcal{AC}(\R)\text{ and } (u'+Bu)'-Bu',x^2u\in\L^2(\R)\right\}.$$

    Moreover, for any $u,v\in \mathrm{D}(A)$ and any $a<b$, the following holds
    $$\int_a^b (-A u) v \d x = W[u,v](b)-W[u,v](a) + \int_a^b u (-A v) \d x,$$
    where $W[u,v](x)=u(x)v'(x)-u'(x)v(x)$ is the Wronskian of $u$ and $v$.

\end{prop}

\begin{proof}

    Let $$ \mathrm{D}_0=\left\{u\in \L^2(\R)\cap\mathcal{AC}(\R),\; u'+Bu\in\mathcal{AC}(\R) \text{ and } (u'+Bu)'-Bu',x^2u\in\L^2(\R)\right\},$$ and define $A_0 u=(u'+Bu)'-Bu'-x^2u$ on $\mathrm{D}_0$. Let $u\in \mathrm{D}(A)\subset\W^{1,2}$, then $u\in H^1(\R)$ and $(u'+Bu)'-Bu'-x^2u = Au\in\L^2(\R)$ in the distributional sense. Thus, $u'+Bu\in\mathcal{AC}(\R)$ and $u\in \mathrm{D}_0$, as \cref{Prop-LocDomain1d} implies $x^2u\in\L^2(\R)$, which implies $A\leqslant A_0$. For $u,v\in \mathrm{D}_0$, a direct integration by part gives
    $$\int_a^b (-A_0 u) v \d x = W[u,v](b)-W[u,v](a) + \int_a^b u (-A_0 v) \d x.$$
    It was shown in \cite{DumazLabbeAndersonLoc1D} that $W[u,v](c)$ goes to 0 when $c$ goes to $\pm\infty$. Sending $a$ to $-\infty$ and $b$ to $+\infty$, we obtain that $A_0$ is symmetric. As $A$ is self-adjoint, it is maximal symmetric, thus $\mathrm{D}(A)=\mathrm{D}_0$.
    
\end{proof}

\subsection{Confining Anderson operator in dimension 2}\label{Sec-AndersonOp2d}

Contrarily to the dimension 1, we cannot use anymore $\xi = B'$ as the two-dimensional counterpart is $\xi=\p_1\p_2 B$ where $(B_{x,y})_{x,y\in\R}$ is now a two-dimensional brownian sheet. We should thus look at the quadratic form approach, yet defining directly 
$$(-Au,v)=\re\int_{\R^2}\nabla u\cdot \overline{\nabla v} + |x|^2u\overline{v} \d x - \langle \xi u, v\rangle$$
is hopeless as $\xi\in\W^{-1-,\infty}$ a.s.\ by \cref{Cor-RegWN}, so this quadratic form is ill-defined on the energy space of $-H$. Thus, we use an exponential transform inspired by the one introduced in \cite{Haire_Labbe} to study the continuous parabolic Anderson model, and later used in \cite{debussche_weber_T2,tzvetkov2020dimensional,Tzvetkov_2023} to solve the Anderson non-linear Schrödinger equation on the torus $\T^2$ and in \cite{debussche2017solution,debussche2023global} to solve the same equation on the full space $\R^2$. Let $Y=(-H)^{-1}\xi$, it is almost surely of regularity $1-$, so one can define $\rho=\e^Y$ and search for a formula for $(H+\xi)u$ when $u$ is of the form $u=\rho v$. A formal computation gives
$$(H+\xi)u=\rho\left(Hv+2\nabla Y\cdot \nabla v + xY\cdot xv + |\nabla Y|^2v\right).$$
Unfortunately, $\nabla Y$ is of regularity $0-$, thus the product $|\nabla Y|^2$ barely fails to exist. The author proved in \cite{Mackowiak_2025} that the wick product $\wick{|\nabla Y|^2}$ can be defined as the almost sure limit, in suitable distribution spaces of regularity $0-$, of
$$|\nabla Y_N|^2-\E[|\nabla Y_N|^2],$$
where $Y_N$ is a suitable regularisation of $Y_N$ with only low modes. 

\begin{lemme}{(Proposition 3.11. in \cite{Mackowiak_2025})}

    Let $p\geqslant 1$, $q\in(2,+\infty]$ and $s>\frac{2}{q}$. Then $(\wick{|\nabla Y_N|^2})_{N\in\N}$ is Cauchy in the Banach space $\L^p(\Omega,\W^{-s,q})$. Moreover, denote by $\wick{|\nabla Y|^2}$ its limit, then almost surely, $\wick{|\nabla Y_N|^2}$ converges to $\wick{|\nabla Y|^2}$ in $\W^{-s,q}$.
    
\end{lemme}

Thus, the renormalized operator can be formally defined as
\begin{equation}
    (H+\xi)u=\rho\left(Hv+2\nabla Y\cdot \nabla v + xY\cdot xv + \wick{|\nabla Y|^2}v\right).\label{Eq-FormulaH+Xi}
\end{equation}
This shows that knowing a trajectory of $\xi$ is not sufficient to define the operator $H+\xi$. It is a usual fact in the study of Anderson operators that one needs more data in order to define the operator (see for example \cite{allez2015continuous,mouzard2021weyl,bailleul2022analysis,hsu2024construction}). In what follows, in order to shorten the notations, we may write $Z=\wick{|\nabla Y|^2}$ and $\Xi=(Y,Z)$ the enhanced noise.\\

The formula \cref{Eq-FormulaH+Xi} does not make sense directly, but as usual, the bilinear form of the operator can be defined rigorously. Recall the space $\D^{1,2}=\rho\W^{1,2}$, as in dimension 1, and endow it with the norm
\begin{equation}\label{Eq-D12nom}
    \forall u=\rho v\in \D^{1,2},\; |u|_{\D^{1,2}}^2=\int_{\R^2}\left(|\nabla v|^2+|xv|^2\right)\rho^2(x)\d x.
\end{equation}
Remark that this norm verifies
\begin{equation}
    \forall v\in \W^{1,2},\; (\inf\rho)|v|_{\W^{1,2}}\leqslant|\rho v|_{\D^{1,2}}\leqslant(\sup\rho)|v|_{\W^{1,2}},\label{Eq-NormEquivD12}
\end{equation}
thus it immediately follows that $(\D^{1,2},|\cdot|_{\D^{1,2}})$ is a Hilbert space. For any $u_i=\rho v_i\in \D^{1,2}$, formally, one can write
$$\langle-(H+\xi)u_1,u_2\rangle=\re \int_{\R^2}\left(\nabla v_1\cdot \overline{\nabla v_2} + |x|^2(1-Y)v_1\overline{v_2}\right)\rho^2\d x - \langle Z, \overline{v_1}v_2\rho^2\rangle.$$
Let define the bilinear form
\begin{equation}
    \forall u_1,u_2\in \D^{1,2},\; a(u_1,u_2)=\re \int_{\R^2}\left(\nabla v_1\cdot \overline{\nabla v_2} + |x|^2(1-Y)v_1\overline{v_2}\right)\rho^2\d x - \langle Z, \overline{v_1}v_2\rho^2\rangle,\label{Eq-QuadraticFormH+xi}
\end{equation}
where $u_i=\rho v_i$. The author proved in \cite{MackowiakStrichartzConfAnderson} this bilinear form is well-defined, symmetric, continuous and quasi-coercive on $\D^{1,2}$.

\begin{prop}{(Proposition 3.4. in \cite{MackowiakStrichartzConfAnderson})}\label[prop]{Prop-Quasi-coercivity-a-2d}

    There exists $C_\Xi,\delta_\Xi>0$ such that
    $$\forall u\in \D^{1,2},\; \frac{1}{2}|u|_{\D^{1,2}}^2\leqslant a(u,u)+\delta_\Xi|u|_{\L^2}^2 \leqslant C_\Xi|u|_{\D^{1,2}}^2.$$
    
\end{prop}

 Hence, Theorem 6.2.6 in \cite{KatoPerturbationBook} implies the existence of a unique lower-bounded self-adjoint operator with compact resolvent $-A:\mathrm{D}(-A)\subset\L^2(\R^2)\to\L^2(\R^2)$ such that 
 $$\forall u_1,u_2\in \D^{1,2},\; a(u_1,u_2) = \langle-Au_1,u_2\rangle_{\D^{-1,2},\D^{1,2}}.$$
Moreover, one can characterize the first Sobolev spaces of $A$.
\begin{prop}{(Proposition 3.5. in \cite{MackowiakStrichartzConfAnderson})}\label[prop]{Prop-CaracDs2}
    For any $s\in(-2,2)$, $\D^{s,2} = \mathrm{D}\left(|A|^{\frac{s}{2}}\right)$ with equivalent norm.
\end{prop}
In particular, $\mathrm{D}(-A)$ is dense and compactly embedded in any $\D^{s,2}$  with $|s|<2$, allowing to make sense of the formal formula \cref{Eq-FormulaH+Xi}.\\

As in dimension 1, when $\xi$ is a white noise, we get a maximal regularity result for elements of the domain, which relies on Sazonov theorem.

\begin{cor}\label[cor]{Cor-MaximalRegDomain2d}
    Almost surely, any $u\in\mathrm{D}(A)$ which belongs to $\mathrm{H}^{1,q}(U)$ for some open set $U\subset\R^2$ and $q\in[2,+\infty]$ vanishes everywhere on $U$.
\end{cor}

\begin{proof}[Proof of \cref{Cor-MaximalRegDomain2d}.]

    We work on the event 
    $$\O_0=\{Y\in\L^{\infty}(\R^2)\}\cap\bigcap_{r\in\mathbb{Q}^2}\bigcap_{n\in\N}\left\{\xi\notin\mathrm{H}^{-1}\left(B(r,2^{-n})\right)\right\}$$
    which is of probability 1 by \cref{Cor-RegWN}, Sazonov theorem and Sobolev embeddings. By \cref{Prop-CaracDs2}, $\mathrm{D}(A)\subset\rho\W^{2-,2}$. Let $v\in\W^{2-,2}$ such that $v_{|U}\neq 0$ and $\rho v\in \mathrm{H}^{1,q}(U)$, there exists $r\in\mathbb{Q}^2$, $n\in\N$ and $c>0$ such that $B=B(r,2^{-n})\subset U$ and
    $$\forall x\in B,\; |v(x)|>c.$$  
    It follows from $\nabla(\rho v)=(v\nabla Y+\nabla v)\rho$ that
    $\nabla Y\in \L^{q}(B)\subset\L^2(B)$ as $\rho^{-1}\in\L^\infty(B)$, $\nabla v\in\W^{1-,2}\subset\L^2(B)$ and $v^{-1}\in\L^\infty(B)$. As $VY\in\L^\infty(B)\subset\L^2(B)$, it holds $\xi\in\mathrm{H}^{-1}(B)$, which is impossible on $\O_0$. Thus $v_{|U}= 0$, i.e.\ $u_{|U}= 0$.
    
\end{proof}

Let $p\in(2,+\infty]$ and $W\in\L^p(\R^2,\R)$. As in dimension 1, one can use the Sobolev injection of $\W^{1-,2}$ in $\L^q(\R^2)$ for any $q\in[2,+\infty)$ and the previous construction to show $(-A+W,\mathrm{D}(A))$ is still a well-defined self-adjoint operator of compact inverse. We want to show that $A+W$ has the spectral gap property, i.e.

\begin{prop}\label[prop]{Prop-SpectralGapA2d}
    Let $p\in(2,+\infty]$ and $W\in\L^p(\R^2,\R)$. Then the lowest eigenvalue of $-A+W$ is simple and the associated eigenfunction can be chosen positive.
\end{prop}

To show $T:\mathrm{D}(T)\subset\L^2(\R^d)\to\L^2(\R^d)$ has a spectral gap, it is a classical consequence of the Krein-Rutman theorem that a sufficient condition is $(\e^{-tT})_{t>0}$ is positive and irreducible, that is
$$\forall f\in\L^2(\R^d),\; f\neq 0,\; f\geqslant 0\; a.e.\; \Rightarrow \forall t>0,\; \e^{-tT}f>0\; a.e.$$
In that case, we can even take the eigenfunction associated to $\lambda_0$ to be positive. The following result, which is a synthesis of Theorems 2.6 and 2.10 of \cite{OuhabazHeatEquation}, gives a characterization of positivity and irreducibility through the quadratic form associated to $T$.

\begin{prop}\label{Prop:PosIrredSemiGroup}
    Let $\mathcal{H}$ be a Hilbert space continuously embedded in $\L^2(\R^d)$. Let $q$ be a continuous, symmetric and coercive bilinear map on $\mathcal{H}$ and $(T,\mathrm{D})$ be the associated self-adjoint operator given by Riesz theorem. Then, $(\e^{-tT})_{t>0}$ is positive and irreductible if and only if there exists dense subsets $\mathcal{H}_0$ and $\mathcal{H}_1$ of $\mathcal{H}$ such that
    \begin{enumerate}
        \item $\forall u\in \mathcal{H}_0,\; (\re u)_+\in \mathcal{H},\; q(\re u,\im u)\in\R$ and $q((\re u)_+,(\re u)_-)\leqslant 0,$
        \item for any $D\subset \R^d$ such that 
        $$\forall u\in \mathcal{H}_1,\; \indic{D}u\in \mathcal{H} \text{ and } q(\indic{D}u,\indic{D^c}u)\geqslant 0,$$
        it holds $|D|=0$ or $|D^c|=0$.
    \end{enumerate}
    In that case, 1. and 2. hold for $\mathcal{H}_0=\mathcal{H}_1=\mathcal{H}$.
\end{prop}

\begin{proof}[Proof of \cref{Prop-SpectralGapA2d}.]

    The quadratic form associated to $-A+W$ is 
    $$\forall u_1,u_2 \in \D^{1,2},\; q(u_1,u_2) = a(u_1,u_2)+\re\int Wu_1\overline{u_2}\d x.$$
    Let $u\in\D^{1,2}$, the diamagnetic inequality implies $(\re u)_+\in\D^{1,2}$ Moreover, one easily check it verifies $q((\re u)_+,(\re u)_-)=0$, $q(\re u, \im u)\in\R$ and $q(u_1,u_2)=0$ if $\supp(u_1)\cap\supp(u_2)=\emptyset$. Hence, according to \cref{Prop:PosIrredSemiGroup}, it is sufficient to show that, if $D\subset\R^2$ verify
    $$\forall u\in \rho\mathcal{C}^\infty_c(\R^2)\subset \D^{1,2},\; \indic{D}u\in\D^{1,2},$$
    then $|D|=0$ or $|D^c|=0$. Let such a $D$, then for all $u=\rho v\in\rho\mathcal{C}^\infty_c(\R^2)$, it holds $\indic{D}u \in\D^{1,2}$ thus $\indic{D}v\in\H^1(\R^2)$. We conclude using the case of the laplacian, as one can easily check $-\Delta+1$ has a positive and irreducible semigroup from the formula
    $$\forall t>0,\; \forall u\in\L^2(\R^2),\; \e^{t(\Delta-1)}u(x) = \e^{-t}\int_{\R^2}\e^{-\frac{|x-y|^2}{4t}}u(y)\frac{\d y}{4\pi t}\; a.e.$$
    The conclusion then follows by a standard Krein-Rutman argument. 
    
\end{proof}

\section{A priori properties of standing waves}\label{Sec-APriori}
\sectionmark{A priori properties}

In this section, we explore a priori results on solutions of \cref{Eq-StationaryAGP} in $\D^{1,2}=\rho\W^{1,2}$. Recall that, in dimension 1, $\D^{1,2}=\W^{1,2}$ and that in dimension 1 and 2, the Riesz theorem implies $-A+\delta:\D^{1,2}\to\D^{-1,2}=(\D^{1,2})'$ is invertible for $\delta$ given by \cref{Prop-Quasi-coercivity-a-1d,Prop-Quasi-coercivity-a-2d}. In both dimensions, $\D^{1,2}$ is the energy domain of $A$. In order to explore the structure of solutions of \cref{Eq-StationaryAGP}, we introduce the following sets
$$\S_{\lambda,\gamma} = \left\{(\omega,u)\in\R\times\D^{1,2},\; -Au-\lambda|u|^{2\gamma}u+\omega u = 0 \text{ and } u\neq 0\right\}$$
and $\overline{\S}_{\lambda,\gamma}$ its closure in $\R\times\D^{1,2}$. Using the continuous embedding of $\D^{1,2}$
in $\L^q(\R^d)$ for any $q\in[2,+\infty)$, it is easy to show
$$\overline{\S}_{\lambda,\gamma} = \S_{\lambda,\gamma} \cup (\R\times\{0\})\subset \R\times\mathrm{D}(A).$$
The set $\R\times\{0\}$ is called the trivial line of solutions of \cref{Eq-StationaryAGP}. Taking into account the exponential transform, one can show 
$$\S_{\lambda,\gamma} = \left\{(\omega,\rho v)\in\R\times\D^{1,2},\; -\Tilde{A}v-\lambda\rho^{2\gamma}|v|^{2\gamma}v+\omega v = 0 \text{ and } v\neq 0\right\},$$
where $$\Tilde{A}v=\rho^{-1}A(\rho v)=Hv+2\nabla Y\cdot\nabla v+ xY\cdot xv + Zv,$$
with $Z=|Y'|^2$ in dimension 1 and $Z=\wick{|\nabla Y|^2}$ in dimension 2. Finally, we define 
$$\Hat{\S}_{\lambda,\gamma} = \left\{u\in\D^{1,2},\; \exists \omega\in\R,\; (\omega,u)\in\S_{\lambda,\gamma}\right\}.$$
As $\Hat{\S}_{\lambda,\gamma}\subset \mathrm{D}(A)$, \cref{Cor-MaximalRegDomain1d,Cor-MaximalRegDomain2d} apply to solutions of \cref{Eq-StationaryAGP}. To prove \cref{Th-MaxRegSW}, it is thus sufficient to prove the following lemma.

\begin{lemme}\label[lemme]{Lem-RegAprioriSW}
    Let $\lambda\in\R$ and $\gamma>0$. A.s. it holds
    $$ \Hat{\S}_{\lambda,\gamma}\subset\bigcap_{q\in[2,+\infty)}\rho\W^{3-\frac{d}{2}-,q}\subset\bigcap_{q\in[2,+\infty)}\W^{2-\frac{d}{2}-,q}.$$
    In particular, a.s.\ $\Hat{\S}_{\lambda,\gamma}\subset\rho\mathcal{C}^{3-\frac{d}{2}-}(\R^d)\subset\mathcal{C}^{2-\frac{d}{2}-}(\R^d).$
\end{lemme}

\begin{proof}

    Let $u\in\Hat{\S}_{\lambda,\gamma}$, there exists $\omega\in\R$ such that $(\omega,u)\in \S_{\lambda,\gamma}$. Write $u=\rho v$ with $v\in\W^{2-,2}$ according to \cref{Prop-LocDomain1d,Prop-CaracDs2}. Then, one has
    $$-Hv=2\nabla Y\cdot \nabla v + xY\cdot xv + Zv + \lambda|v|^{2\gamma}v\rho^{2\gamma}-\omega v,$$
    where $Z=|Y'|^2$ in dimension 1 and $Z=\wick{|\nabla Y|^2}$ in dimension 2. We work on the event $$\{Y\in \H^{2-\frac{d}{2}-,\infty},\; \nabla Y,xY, Z\in \H^{1-\frac{d}{2}-,\infty}\}$$ which is of full probability. As $u\in \mathrm{D}(A)\subset \rho \W^{2-\frac{d}{2}-,q}$ for any $q\in[2,+\infty)$ by Sobolev embeddings, \cref{Lem-ProductRule,Cor-ProductRuleNegPosReg,Prop-action_d/dx,Cor-action_V} imply $v\in \W^{3-\frac{d}{2}-,q}$ for any $q\in[2,+\infty)$. Hölder regularity follows from usual Sobolev embeddings.
    
\end{proof}

Using elementary computations, we can prove a priori conditions on $\omega$ for the existence of non-negative standing waves.

\begin{prop}\label[prop]{Prop-APrioriConditionExistence}
    Let $\omega\in \R$. There are three cases : 
    \begin{itemize}
        \item If $\lambda=0$, there exists $u\geqslant0$ with $(\omega,u)\in\S_{0,\gamma}$ if and only if $\omega=-\mu_0$ and in that case, $u=\alpha\varphi_0$ ($\alpha>0$).
        \item If $\lambda>0$ and there exists $u\geqslant0$ with $(\omega,u)\in\S_{\lambda,\gamma}$ then $\omega>-\mu_0$. If $\omega>-\mu_0$ and there exists $u\geqslant0$ with $(\omega,u)\in\S_{\lambda,\gamma}$ then $\lambda>0$.
        \item If $\lambda<0$ and there exists $u\geqslant0$ with $(\omega,u)\in\S_{\lambda,\gamma}$ then $\omega<-\mu_0$. If $\omega<-\mu_0$ and there exists $u\geqslant0$ with $(\omega,u)\in\S_{\lambda,\gamma}$ then $\lambda<0$.
    \end{itemize}  
\end{prop}

\begin{proof}

    The case $\lambda=0$ comes from \cref{Prop-SpectralGapA1d,Prop-SpectralGapA2d}. Now, assume $\omega>-\mu_0$ and $u\geqslant 0$ with $(\omega,u)\in\S_{\lambda,\gamma}$, taking a duality bracket of \cref{Eq-StationaryAGP} with $u$, it holds
    $$0<\langle (-A+\omega)u,u\rangle = \lambda \int_{\R^d}|u|^{2\gamma+2}\d x,$$
    thus $\lambda>0$. Conversely, if $\lambda>0$ and $u$ as previously, taking a duality bracket of \cref{Eq-StationaryAGP} with $\varphi_0$, it holds
    $$0<\lambda \int_{\R^d}|u|^{2\gamma}u\varphi_0\d x = \langle (-A+\omega)u,\varphi_0\rangle = (\omega+\mu_0)(u,\varphi_0)_{\L^2}.$$
    As $\varphi_0>0$ and $u\geqslant 0$ is non-zero, it holds $\omega>-\mu_0$. Similarly, if $\lambda<0$, 
    $$0>\lambda \int_{\R^d}|u|^{2\gamma}u\varphi_0\d x = \langle (-A+\omega)u,\varphi_0\rangle = (\omega+\mu_0)(u,\varphi_0)_{\L^2},$$
    thus $\omega<-\mu_0$. Finally, if $\omega<-\mu_0$ and $u\geqslant 0$ with $(\omega,u)\in\S_{\lambda,\gamma}$, previous results show $\lambda<0$ as it cannot be non-negative because of $\omega$.
    
\end{proof}

We now prove the a priori strict positivity of solutions.

\begin{prop}\label[prop]{Prop-Apriori>0}
    Let $(\omega,u)\in\S_{\lambda,\gamma}$ with $u\geqslant 0$. Then, $u>0$ and for every $(\omega,\phi)\in\S_{\lambda,\gamma}$ with $|\phi|=u$, there exists $\theta>0$ such that $\phi=\e^{i\theta}u$.
\end{prop}

\begin{proof}

    Let $W=\omega-\lambda u^{2\gamma}\in\L^\infty(\R^d)$ according to \cref{Lem-RegAprioriSW}. It holds $$(-A+W)u=0,$$ so $u$ is an eigenfunction of $-A+W$. As $u\geqslant 0$, \cref{Prop-SpectralGapA1d,Prop-SpectralGapA2d} imply it must be associated to the lowest eigenvalue, $u>0$ and every other solution of $(-A+W)\phi=0$ is of the form $\phi=\alpha u$ with $\alpha\in\C$. Thus if $(\omega,\phi)\in\S_{\lambda,\gamma}$ with $|\phi|=u$, there exists $\theta>0$ such that $\phi=\e^{i\theta}u$.
    
\end{proof}

\subsection{Exponential localization}

In this section, we will extend the exponential localization result from \cite{FukuizumiExponentialDecay} to focusing standing waves with white noise potential. This approach is mostly perturbative. Let $\eps\in(0,1)$ and $\phi\in\mathcal{C}^1(\R^d\backslash\{0\},\R)\cap\mathcal{C}(\R^d,(0,+\infty))$ verifying
\begin{equation}\label{Eq-PhiLocExpo}
    \forall x\in\R^d\backslash\{0\},\; |\nabla \phi(x)|^2\leqslant 4(1-\eps)^2 V(x).
\end{equation}

\begin{prop}\label[prop]{Prop-LocalExpoFoca}

    Let $\lambda,\gamma>0$. There exists $c=c(\Xi,\eps,V)>0$, such that for all $(\omega,u)\in\S_{\lambda,\gamma}$ with $u=\rho v$, $\e^\phi\left(|\nabla v|^2+V|v|^2+|v|^{2\gamma+2}\right)\in\L^1(\R^d)$ and 
    \begin{align*}
        &\frac{\eps}{2}\int \e^\phi \left(|\nabla v|^2+|xv|^2\right) \rho^2 \d x + (\omega-c)\int \e^\phi |v|^2\rho^2 \d x\\ 
        &\leqslant \lambda \int \e^\phi |v|^{2\gamma+2}\rho^{2\gamma+2} \d x\\ 
        &\leqslant \left(2-\frac{\eps}{2}\right)\int \e^\phi \left(|\nabla v|^2+|xv|^2\right)\rho^2 \d x + (\omega+c)\int \e^\phi |v|^2\rho^2 \d x.
    \end{align*}
    
\end{prop}

\begin{proof}
    We do the proof for dimension 2 as the dimension 1 is easier. As in the proof of Theorem 2 in \cite{FukuizumiExponentialDecay}, for $m\in\N$, let
    $$\zeta_m = \exp(m\phi/(m+\phi)).$$
    We show there exists $c>0$ such that
     \begin{align*}
        &\frac{\eps}{2}\int \zeta_m \left(|\nabla v|^2+|xv|^2\right) \rho^2 \d x + (\omega-c)\int \zeta_m |v|^2 \rho^2\d x\\ 
        &\leqslant \lambda \int \zeta_m |v|^{2\gamma+2} \rho^{2\gamma+2}\d x\\ 
        &\leqslant \left(2-\frac{\eps}{2}\right)\int \zeta_m \left(|\nabla v|^2+|xv|^2\right)\rho^2 \d x + (\omega+c)\int \zeta_m |v|^2\rho^2 \d x.
    \end{align*}
    The end of the proof follows exactly the one of Theorem 2 in \cite{FukuizumiExponentialDecay}, by sending $m$ to $+\infty$,
    as $V$ verifies hypothesis (V3) of \cite{FukuizumiExponentialDecay} and $u$ tends to 0 at infinity.\\

    Taking a duality bracket of \cref{Eq-StationaryAGP} with $\zeta_m u = \zeta_m \rho v$ and using 
    $$A(\rho v)=\rho\left(Hv+2\nabla Y\cdot \nabla v + xY\cdot xv + Zv\right),$$
    it holds
    \begin{equation}
        \begin{split}
            &\re\int_{\R^d}\rho^2 \nabla v\cdot\overline{\nabla[\zeta_m v]}\d x + \int_{\R^d} \zeta_m (|x|^2+\omega)|v|^2 \rho^2 \d x\\
            &= \lambda \int_{\R^d} \zeta_m |v|^{2\gamma+2} \rho^{2\gamma+2}\d x + \langle |x|^2Y+Z, \zeta_m|v|^2\rho^2\rangle.
        \end{split}
    \end{equation}
    Now, remark
    $$\re\int_{\R^d}\rho^2 \nabla v\cdot\overline{\nabla[\zeta_m v]}\d x = \int_{\R^d}\zeta_m|\nabla v|^2\rho^2\d x + \re\int_{\R^d}\left(\frac{m}{m+\phi}\right)^2\zeta_m v \nabla\phi\cdot\overline{\nabla v}\d x.$$
    Thus, using Young inequality and \cref{Eq-PhiLocExpo}, it holds
    \begin{align*}
        &\eps\int \zeta_m \left(|\nabla v|^2+|xv|^2\right) \rho^2 \d x + \omega\int \zeta_m |u|^2 \d x\\ 
        &\leqslant \lambda \int \zeta_m |u|^{2\gamma+2} \d x + \langle VY+Z, \zeta_m|u|^2\rangle\\ 
        &\leqslant \left(2-\eps\right)\int \zeta_m \left(|\nabla v|^2+|xv|^2\right)\rho^2 \d x + \omega\int \zeta_m |u|^2 \d x.
    \end{align*}
    To conclude, it is sufficient to treat $\langle VY+Z, \zeta_m|u|^2\rangle$ as a lower order perturbation. Recall $xY,Z\in\W^{0-,\infty}$, thus by duality and \cref{Lem-ProductRule,Cor-ProductRuleNegPosReg}, it holds for any $\kappa\in(0,\frac{3}{4})$, 
    $$|\langle |x|^2Y+Z, \zeta_m|v|^2\rho^2\rangle| \lesssim |xY|_{\W^{-\kappa,\infty}}|\rho|_{\H^{\kappa,\infty}}^2\left|\sqrt{|x|}\zeta_m^{\frac{1}{2}}v\right|_{\W^{\kappa,2}}^2+|Z|_{\W^{-\kappa,\infty}}|\rho|_{\H^{\kappa,\infty}}^2\left|\zeta_m^{\frac{1}{2}}v\right|_{\W^{\kappa,2}}^2.$$
    Thus, by \cref{Cor-action_V}, interpolation and Young inequality, for any $\eta>0$, there exists $k=k(\eta,\Xi)>0$ such that,
    $$|\langle |x|^2Y+Z, \zeta_m|v|^2\rho^2\rangle| \leqslant \eta \left|\zeta_m^{1/2}v\right|_{\W^{1,2}}^2 + k\int_\R \zeta_m |v|^2 \rho^2\d x.$$

    Moreover,
    \begin{align*}
        \left|\zeta_m^{1/2}v\right|_{\W^{1,2}}^2 &\leqslant 2 \int \zeta_m|\nabla v|^2\d x + \frac{1}{2}\int \zeta_m\frac{m^2|\nabla\phi|^2}{(m+\phi)^2}|v|^2\d x + \int \zeta_m |xv|^2\d x \\
        &\leqslant 2 \int \zeta_m|\nabla v|^2\d x + \left(2(1-\eps)^2+1\right) \int \zeta_m |xv|^2\d x \\
        &\leqslant 3\left|\rho^{-2}\right|_{\L^\infty} \int \zeta_m \left(|\nabla v|^2+|xv|^2\right) \rho^2 \d x.
    \end{align*}
    Choosing $\eta = \frac{\eps}{6\left|\rho^{-2}\right|_{\L^\infty}}$ allows to conclude.
    
\end{proof}

As a corollary, we obtain pointwise exponential localization of $u$ and $\nabla v$. 

\begin{cor}\label[cor]{Cor-PointWiseExp}
    Let $\lambda,\gamma>0$. Then, for all $u=\rho v\in\Hat{\S}_{\lambda,\gamma}$, it holds
    $$\exists C,c>0,\; \forall x\in\R^d,\; |u(x)|+|\nabla v(x)|\leqslant C\e^{-c|x|^2}.$$
\end{cor}

\begin{proof}

    Let $u\in\Hat{\S}_{\lambda,\gamma}$. As $u=\rho v$ with $\rho\in\L^\infty(\R^d)$, it is sufficient to show
    $$\exists C,c>0,\; \forall x\in\R^d,\; |v(x)|+|\nabla v(x)|\leqslant C\e^{-c|x|^{2}}.$$
    Remark
    $$\phi(x)=2(1-\eps)^2|x|^2$$
    verifies (\ref{Eq-PhiLocExpo}). Let $\alpha\in(0,1)$, as a consequence of \cref{Lem-RegAprioriSW}, there exists $L>0$ such that
    $$\forall x,y\in\R^d,\; |v(x)-v(y)|+|\nabla v(x)-\nabla v(y)|\leqslant L|x-y|^\alpha.$$
    We only prove exponential decay of $v$, as the exact same proof applies to $\nabla v$. Let $x\in\R^d$ and set $B(x) = \left\{y\in\R^d,\; |x-y|^\alpha\leqslant\frac{|v(x)|}{2L}\right\}$. Then, for any $y\in B(x)$, it holds
    $$|v(x)|\leqslant 2|v(y)|.$$
    In particular, integrating over $B(x)$, it holds
    \begin{equation}\label{Eq-IntB(x)ExpLoc}
        C_d \left(\frac{|v(x)|}{2L}\right)^{\frac{d}{\alpha}}|v(x)|^2\leqslant 4\int_{B(x)}|v(y)|^2\d y,
    \end{equation}
    where $C_d$ is the measure of the unit ball in $\R^d$. By \cref{Lem-RegAprioriSW}, $v\in\L^\infty(\R^d)$, thus $B(x)\subset B(x,R)$ with $R^\alpha=\frac{|v|_{\L^\infty}}{2L}$.
    Now, remark there exists $c>0$ such that
    $$\forall x,y\in\R^d,\; |\phi(x)-\phi(y)|\leqslant c\phi(x-y),$$
    thus there exists $C>0$ such that for any $x\in\R$ and any $y\in B(x)$, 
    \begin{equation}\label{Eq-ControlexpPhiExpLoc}
        \e^{\phi(x)}\leqslant \e^{\phi(y)}\e^{c\phi(x-y)}\leqslant C \e^{\phi(y)}.
    \end{equation}
    Multiplying \cref{Eq-IntB(x)ExpLoc} by $\e^{\phi(x)}$, \cref{Eq-PhiLocExpo,Eq-ControlexpPhiExpLoc,Prop-LocalExpoFoca} imply
    $$|v(x)|^{2+\frac{d}{\alpha}}\leqslant \frac{4C(2L)^{\frac{d}{\alpha}}}{C_d}\int_{\R^d}\e^{\phi(y)}|v(y)|^2\d y \e^{-\phi(x)}\lesssim \e^{-\phi(x)}.$$
    As $\phi(x)=2(1-\eps)^2|x|^2$, there exists $c'>0$ such that
    $$|v(x)|\lesssim\e^{-\frac{\phi(x)}{2+\frac{d}{\alpha}}}\lesssim \e^{-\frac{2\alpha(1-\eps)^2}{2\alpha+d}|x|^{2}},$$
    which allows to conclude.

\end{proof}

Let $u=\rho v\in\Hat{\S}_{\lambda,\gamma}$ and $w(x)=\exp\left(c'|x|^{2}\right)$ with $c>c'>0$ for $c$ given by \cref{Cor-PointWiseExp}. Then $wv\in\L^\infty(\R^d)$ and
$$|\nabla(wv)(x)|\lesssim w(x)|\nabla v(x)|+ |xv(x)| w(x) \lesssim \left(|v(x)|+|\nabla v(x)|\right)\exp\left(c|x|^{2}\right)\lesssim 1 $$
by \cref{Cor-PointWiseExp}. Thus $wv\in\H^{1,\infty}(\R^d)$ and $wu\in\H^{1-,\infty}(\R^d)$ by \cref{Lem-ProductRule}. In particular, in dimension 2, using
$$\nabla u = u\nabla Y + \rho\nabla v\in\mathcal{C}^{0-},$$ 
we get $w\nabla u\in\mathcal{C}^{0-}$. This shows $\nabla u$ is still exponentially localized even if it lacks regularity.

\section{Energy ground-states}\label{Sec-EnergyGS}

We now explore the properties of energy ground-states. Recall the energy is defined by 
$$\forall u\in\D^{1,2},\; \mathcal{E}(u)= \frac{1}{2}\langle -Au,u\rangle - \frac{\lambda}{2\gamma+2}\int_{\R^d}|u|^{2\gamma+2}\d x$$
and the mass is defined by 
$$\forall u\in\L^2(\R^d),\; \mathcal{M}(u)= \frac{1}{2}|u|_{\L^2}^2.$$
Moreover, for $u=\rho v\in\D^{1,2}$, one can rewrite the energy as
\begin{equation}\label{Eq-EnergyVside}
    \mathcal{E}(u) = \frac{1}{2}\left(\int_{\R^d} |\rho\nabla v|^2+(1-Y)|\rho xv|^2\d x - \langle Z,|\rho v|^2\rangle\right) - \frac{\lambda}{2\gamma+2}\int_{\R^d}|\rho v|^{2\gamma+2}\d x,
\end{equation}
with $Y=(-H)^{-1}\xi$, $Z=(Y')^2$ in dimension 1 and $Z=\wick{|\nabla Y|^2}$ in dimension 2. For any $m>0$, we define the minimal energy level as
$$E(m) = \inf_{\Atop{u\in\D^{1,2}}{\mathcal{M}(u)=m}}\mathcal{E}(u).$$
If it is finite, we say that $u\in\D^{1,2}$ is a ground-state of the energy at mass $m$ if $\mathcal{E}(u)=E(m)$ and $\mathcal{M}(u)=m.$ A standard Lagrange multiplier argument then implies it exists $\omega_u\in\R$ such that $(\omega_u,u)\in\S_{\lambda,\gamma}$.

\subsection{Existence of energy ground-states}

First, we prove the existence of an energy ground-state in the defocusing case or in the focusing subcritical case. It relies on the following lemma, allowing to use compactness methods.

\begin{lemme}\label[lemme]{Lem-ControlD12Energy}
    Let $\lambda\in\R$ and $\gamma>0$. If $\lambda>0$, assume $\gamma<\frac{2}{d}$. There exists $C>0$ such that
    $$\forall u\in\D^{1,2},\; |u|_{\D^{1,2}}^2\leqslant C\left(\mathcal{E}(u)+\mathcal{M}(u)+\lambda_+ \mathcal{M}(u)^{1+\frac{\gamma}{2-d \gamma}}\right).$$
\end{lemme}

\begin{proof}

    Let $u\in\D^{1,2}$ and denote $m=\mathcal{M}(u)$. Applying \cref{Prop-Quasi-coercivity-a-1d,Prop-Quasi-coercivity-a-2d} one has
    \begin{align*}
        \frac{1}{2}|u|_{\D^{1,2}} &\leqslant \langle -A u, u\rangle +  \delta |u|_{\L^2}^2\\
        &\leqslant 2\mathcal{E}(u)+2\delta m + \frac{\lambda}{2\gamma+2}\int_{\R^d}|u|^{2\gamma+2}\d x.
    \end{align*}
    Thus, if $\lambda \leqslant 0$, the result holds with $C=4\max(1,\delta)$. 
    If $\lambda>0$ and $\gamma<\frac{2}{d}$, then we use Sobolev embeddings and interpolation to get
    $$\int_{\R^d}|u|^{2\gamma+2}\d x \lesssim |u|_{\W^{\frac{d\gamma}{2\gamma+2},2}}^{2\gamma+2} \lesssim  |u|_{\D^{1,2}}^{d\gamma}|u|_{\L^2}^{(2-d)\gamma+2},$$
    as $\frac{d\gamma}{2\gamma+2}\in(0,1)$, thus \cref{Lem-ProductRule} implies $\W^{\frac{d\gamma}{2\gamma+2},2}=\D^{\frac{d\gamma}{2\gamma+2},2}$. By Young's inequality, as $d\gamma<2$, there exists $c>0$ such that
    $$\frac{1}{2\gamma+2}\int_{\R^d}|u|^{2\gamma+2}\d x \leqslant \frac{1}{4\lambda}|u|_{\D^{1,2}}^2+cm^{1+\frac{\gamma}{2-d \gamma}}.$$
    Hence, we get
    $$\frac{1}{4}|u|_{\D^{1,2}}^2\leqslant 2\mathcal{E}(u)+2\delta m + \lambda c m^{1+\frac{\gamma}{2-d \gamma}}$$
    and we conclude choosing $C=\max(8,8\delta,4c)$.
    
\end{proof}

\begin{prop}\label[prop]{Prop-ExistenceEnergyGS}
    Assume $\lambda\leqslant 0$ or $\gamma<\frac{2}{d}$. Then, for any $m>0$, there exists $\phi_m\in\D^{1,2}$ such that $\phi_m\geqslant 0$, $\mathcal{M}(\phi_m)=m$ and $\mathcal{E}(\phi_m)=E(m)$. Moreover, if $\lambda<0$, such a $\phi_m$ is unique and every non-negative minimizing sequence converges to $\phi_m$ in $\D^{1,2}$.
\end{prop}

\begin{proof}

    We use a standard compactness method to construct $\phi_m$. First, remark for any $u=\rho v\in\D^{1,2}$, the diamagnetic inequality implies $|u|=\rho|v|\in\D^{1,2}$ and $\mathcal{E}(|u|)\leqslant \mathcal{E}(u)$. Thus, we only need to construct a $u\in\D^{1,2}$ with $\mathcal{E}(u)=E(m)$ and $\mathcal{M}(u)=m$. Let $m>0$ and $(u_n)_{n\in\N}\subset\D^{1,2}$, with $\mathcal{M}(u_n)=m$, be a minimizing sequence of the energy. Then, applying \cref{Lem-ControlD12Energy} one has
    $$|u_n|_{\D^{1,2}}^2\leqslant C\left(\mathcal{E}(u_n)+ m + \lambda_+ m^{1+\frac{\gamma}{2-d \gamma}}\right).$$
    Thus, $(u_n)_{n\in\N}$ is bounded in $\D^{1,2}$ and Rellich's theorem implies there exists $u\in\D^{1,2}$ such that, up to a subsequence, $u_n$ converges to $u$, weakly in $\D^{1,2}$ and strongly in $\L^2(\R^d)\cap\L^{2\gamma+2}(\R^d)$.
    Hence, $\mathcal{M}(u)=m$ and by sequential weak lower semicontinuity of quadratic forms, 
    \begin{align*}
        E(m) \leqslant \mathcal{E}(u) &= \frac{1}{2}\langle -Au,u\rangle - \frac{\lambda}{2\gamma+2}\int_{\R^d}|u|^{2\gamma+2}\d x\\
        &\leqslant \liminf_{n\to\infty}\frac{1}{2}\langle -Au_n,u_n\rangle - \lim_{n\to\infty}\frac{\lambda}{2\gamma+2}\int_{\R^d}|u_n|^{2\gamma+2}\d x\\
        &\leqslant \limsup_{n\to+\infty} \mathcal{E}(u_n) = E(m).
    \end{align*}
    This concludes the existence part.\\
    
    Now, assume $\lambda<0$ and let $\phi,\psi$ be non-negative ground-states of the energy at mass $m$. Define $f=\sqrt{\frac{\phi^2+\psi^2}{2}}$. Then, the diamagnetic inequality gives
    $$\langle -Af,f\rangle \leqslant \frac{1}{2}\left(\langle -A\phi,\phi\rangle+\langle -A\psi,\psi\rangle\right)$$
    and moreover,
    $$\int_{\R^d}|f|^{2\gamma+2}\d x \leqslant \frac{1}{2}\left(\int_{\R^d}|\phi|^{2\gamma+2}\d x+\int_{\R^d}|\psi|^{2\gamma+2}\d x\right).$$
    Thus, 
    $$E(m) \leqslant \mathcal{E}(f) \leqslant \frac{1}{2}\left(\mathcal{E}(\phi)+\mathcal{E}(\psi)\right) = E(m).$$
    In particular, 
    $$\int_{\R^d}|f|^{2\gamma+2}\d x = \frac{1}{2}\left(\int_{\R^d}|\phi|^{2\gamma+2}\d x+\int_{\R^d}|\psi|^{2\gamma+2}\d x\right),$$
    that is
    $$\left|\phi^2+\psi^2\right|_{\L^{\gamma+1}} = 2^{\frac{\gamma}{\gamma+1}}\left(\left|\phi^2\right|_{\L^{\gamma+1}}^{\gamma+1}+\left|\psi^2\right|_{\L^{\gamma+1}}^{\gamma+1}\right)^{\frac{1}{\gamma+1}} \geqslant \left|\phi^2\right|_{\L^{\gamma+1}}+\left|\psi^2\right|_{\L^{\gamma+1}}.$$
    Thus, there exists $\alpha\in\C$ such that $\phi^2=\alpha\psi^2$ a.e. As $\phi,\psi\geqslant 0$ and $\mathcal{M}(\phi)=\mathcal{M}(\psi)$, then $\phi=\psi$ a.e. We conclude $\phi=\psi$, as standing waves are continuous by \cref{Lem-RegAprioriSW}.\\
    
    Finally, assume $u_n\geqslant 0$ is a minimizing sequence. Then, up to a subsequence, there exists some $u\in\D^{1,2}$ such that $u_n$ converges to $u$, weakly in $\D^{1,2}$ and strongly in $\L^2(\R^d)\cap\L^{2\gamma+2}(\R^d)$. \cref{Lem-RegAprioriSW} implies $u$ is continuous. Assume $u$ is not non-negative, then there exists there exists a ball $B\subset\R^d$ and $\eps<0$ such that $u_{|B}\leqslant -\eps$. Then, as $u_n$ converges to $u$ in $\L^2(\R^d)$, $(u_n,\indic{B})_{\L^2}$ converges to $(u,\indic{B})_{\L^2}\leqslant-\eps|B|<0$. As $u_n$ is non-negative for every $n\in\N$ it follows, it follows $(u_n,\indic{B})_{\L^2}\geqslant 0$, which contradicts the fact that $(u,\indic{B})_{\L^2}<0$. Hence, $u\geqslant 0$ and by uniqueness of the non-negative ground-state, $u_n$ converges to $\phi_m$, weakly in $\D^{1,2}$ and strongly in $\L^2(\R^d)$ and $\L^{2\gamma+2}(\R^d)$, along the whole sequence. Finally, as $\mathcal{E}(u_n)$ converges to $E(m)=\mathcal{E}(\phi_m)$, we deduce that $u_n$ converges strongly to $\phi_m$ in $\D^{1,2}$.

\end{proof}

It is classical that $\displaystyle\argmin_{u\in\D^{1,2},\; \mathcal{M}(u)=m}\mathcal{E}(u)$ is stable in $\D^{1,2}$ in the following sense. Given a $u_0\in\D^{1,2}$, we denote by $u(t)$ the unique local solution of \cref{Eq-AGP} starting from $u_0$, which has been constructed in \cite{MackowiakStrichartzConfAnderson}. We say a set $E\subset\D^{1,2}$ is stable if any local solution starting from $E$ is global and
$$\forall \eps>0,\; \exists \delta>0,\; \forall u_0\in \D^{1,2},\; 
\inf_{f\in E}|u_0-f|_{\D^{1,2}}\leqslant \delta \Rightarrow \sup_{t\in\R}\inf_{f\in E}|u(t)-f|_{\D^{1,2}}\leqslant \eps.$$
The proof of the stability of $\displaystyle\argmin_{u\in\D^{1,2},\; \mathcal{M}(u)=m}\mathcal{E}(u)$ in $\D^{1,2}$ follows by the exact same contradiction argument as in the classical case \cite{CazenaveLionsStability}, using energy and mass conservation of solutions. When $E$ of the form $\left\{\e^{i\theta}\phi,\; \theta\in\R\right\}$ is stable, we say that $\phi$ is orbitally stable.\\

For each $\phi\in\displaystyle\argmin_{u\in\D^{1,2},\; \mathcal{M}(u)=m}\mathcal{E}(u)$, there exists $\omega_\phi\in\R$ such $(\omega_\phi,\phi)\in\S_{\lambda,\gamma}$. A direct computation gives
\begin{equation}\label{Eq-Omega(phi)}
    \omega_\phi = - \frac{E(m)}{m}+\frac{\lambda\gamma}{2(\gamma+1)m}\int_{\R^d}|\phi|^{2\gamma+2}\d x.
\end{equation}
Remark it only depends on $|\phi|$. Thus, using \cref{Prop-Apriori>0} for $|\phi_m|$, we immediately deduce $\phi=\e^{i\theta}|\phi|$, hence the orbital stability of the defocusing energy ground-state.

\begin{cor}\label[cor]{Cor-OrbitalStabDefocEGS}
    Assume $\lambda<0$ and $\gamma>0$, then for any $m>0$,
    $$\argmin_{\Atop{u\in\D^{1,2}}{\mathcal{M}(u)=m}}\mathcal{E}(u) = \left\{\e^{i\theta}\phi_m,\; \theta\in\R\right\}.$$
    In particular, $\phi_m$ is orbitally stable.
\end{cor}

\subsection{Dependence on the mass}

We now explore how $\phi_m$ behaves as $m$ varies. First, it is easy to show the continuity of the minimal energy as a function of mass.

\begin{lemme}\label[lemme]{Lem-ContinuityE(m)}
    Assume $\lambda\leqslant 0$ or $\gamma<\frac{2}{d}$. Then, $m\in(0,+\infty)\mapsto E(m)\in\R$ is continuous.
\end{lemme}

\begin{proof}
    For any $m>0$, let $\phi_m\in\displaystyle\argmin_{u\in\D^{1,2},\; \mathcal{M}(u)=m}\mathcal{E}(u)$. 
    \begin{itemize}
        \item Assume $\lambda\leqslant 0$, for any $m,m'>0$, taking $\sqrt{\frac{m}{m'}}\phi_{m'}$ as a test function implies
        $$E(m)\leqslant \max\left(\frac{m}{m'},\left(\frac{m}{m'}\right)^{\gamma+1}\right)E(m').$$
        It follows $E(m')$ converges to $E(m)$ as $m'$ goes to $m$.
        \item Now, if $\lambda>0$, for any $m,m'>0$, taking again $\sqrt{\frac{m}{m'}}\phi_{m'}$ as a test function implies
        \begin{equation}\label{Eq-ControlE(m)parE(m')Focalisant}
            E(m)\leqslant \left(\frac{m}{m'}\right)^{\gamma+1}E(m') + \frac{1}{2}\left(\frac{m}{m'}-\left(\frac{m}{m'}\right)^{\gamma+1}\right)\langle -A\phi_{m'},\phi_{m'}\rangle.
        \end{equation}
        Moreover, \cref{Lem-ControlD12Energy,Eq-ControlE(m)parE(m')Focalisant} imply $\langle -A\phi_{m'},\phi_{m'}\rangle$ is bounded for $m'$ in a neighborhood of $m$, thus \cref{Eq-ControlE(m)parE(m')Focalisant} implies  $E(m')$ converges to $E(m)$ as $m'$ goes to $m$.
    \end{itemize} 

\end{proof}

When $\lambda=0$, the Spectral theorem implies $E(m)=m\mu_0$, thus $\omega_{\phi}=-\mu_0$, where $\mu_0$ is the lowest eigenvalue of $-A$. \cref{Prop-SpectralGapA1d,Prop-SpectralGapA2d} imply $\phi_{m} = \sqrt{2m}\e^{i\theta}\varphi_0$ for some $\theta\in\R$, with $\varphi_0>0$ the eigenfunction of $-A$ associated to $\mu_0$. From this, we get the following asymptotic for $\omega_{\phi}$ when $\lambda\neq 0$.

\begin{prop}\label[prop]{Prop-FrequencyEnergyGS}
    Let $m>0$ and $\phi_m\in\displaystyle\argmin_{u\in\D^{1,2},\;\mathcal{M}(u)=m}\mathcal{E}(u)$. Then $\omega_{\phi_m}+\mu_0$ has the same sign as $\lambda$ and $$\omega_{\phi_m}+\mu_0=O\left(m^\gamma\right)$$ as $m$ goes to $0$. Moreover, if $\lambda<0$ (respectively $\lambda>0$), $\omega_{\phi_m}$ goes to $-\infty$ (respectively $+\infty$) as $m$ goes to infinity.
\end{prop}

The defocusing case relies on the following interpolation lemma, whose proof is similar to the one of Proposition 6.1. (iv) in \cite{KavianSelfSimilar}, interpolating the $\L^2$-norm between the $\L^p$-norm and some weighted $\L^2$-norm.

\begin{lemme}\label[lemme]{Lem-InterMassEnergNonlin}
    For any $\eps>0$, $\omega\in\R$ and $p\in(2,+\infty]$, there exists $C=C(\eps,\omega,p)>0$ such that,
    $$\forall u\in\D^{1,2},\; \mathcal{M}(u)\leqslant \eps\langle (-A+\omega)u,u\rangle + C|u|_{\L^{p}}^2.$$
\end{lemme}

\begin{proof}

    First, we can reduce to the case $\omega=\delta$ given by \cref{Prop-Quasi-coercivity-a-1d,Prop-Quasi-coercivity-a-2d}. In fact, if the result holds for $\delta$, it holds for any $\omega\geqslant\delta$. Assume $\omega<\delta$ and the result holds for $\delta$. Then, for $\eta>0$, there exists $C>0$ independent of $u$ and which may vary from line to line, such that
    \begin{align*}
        \forall u\in\D^{1,2},\; \mathcal{M}(u)&\leqslant \eta\langle (-A+\delta)u,u\rangle + C|u|_{\L^{p}}^2\\
        &\leqslant\eta\langle (-A+\omega)u,u\rangle + 2\eta(\delta-\omega)\mathcal{M}(u) + C|u|_{\L^{p}}^2,
    \end{align*}
    with $\delta-\omega>0$. Thus, for $\eta=\frac{\eps}{1+2\eps(\delta-\omega)}$, it holds
    $$\forall u\in\D^{1,2},\; \mathcal{M}(u)\leqslant \eps\langle (-A+\omega)u,u\rangle + C|u|_{\L^{p}}^2.$$
    Thus, we only need to prove the case $\omega=\delta$. Let $R>0$, then it holds
    \begin{align*}
        \int_{\R^d}|u(x)|^2\d x &\leqslant \frac{1}{R^2} \int_{|x|>R}|xu(x)|^2\d x + \int_{|x|\leqslant R}|u(x)|^2\d x\\
        &\leqslant\frac{1}{R^2} \int_{|x|>R}|xu(x)|^2\d x + C_d R^{d\frac{p-2}{p}}|u|_{\L^p}^2,
    \end{align*}
    by Hölder's inequality, where $C_d$ is the measure of the unit ball in $\R^d$. Now, \cref{Prop-Quasi-coercivity-a-1d,Prop-Quasi-coercivity-a-2d} imply
    $$\mathcal{M}(u) \leqslant \frac{1}{R^2} \langle (-A+\delta)u,u\rangle + \frac{C_d}{2} R^{d\frac{p-2}{p}}|u|_{\L^p}^2.$$
    Thus, for $\eps>0$ and $R=\eps^{-\frac{1}{2}}$, we get
    $$\mathcal{M}(u) \leqslant \eps \langle (-A+\delta)u,u\rangle + \frac{C_d}{2} R_{\eps}^{d\frac{p-2}{p}}|u|_{\L^p}^2,$$
    which gives the claim for $\omega=\delta$.
    
\end{proof}

\begin{proof}[Proof of \cref{Prop-FrequencyEnergyGS}.]

    First, we prove inequalities for $E(m)$. Taking $u=\sqrt{2m}\varphi_0$, optimality gives
    \begin{equation}
        E(m)\leqslant \mu_0m - \frac{\lambda (2m)^{\gamma+1}}{2\gamma+2}\int_{\R^d}|\varphi_0|^{2\gamma+2}\d x.\label{Eq-UpBoundE(m)}
    \end{equation}
    Now, as $\mu_0$ is the smallest eigenvalue of $-A$, it holds
    \begin{equation}
        E(m)\geqslant \mu_0 m - \frac{\lambda}{2\gamma+2}\int_{\R^d}|\phi_m|^{2\gamma+2}\d x,\label{Eq-LowBoundE(m)}
    \end{equation}
    thus \cref{Eq-Omega(phi),Eq-UpBoundE(m),Eq-LowBoundE(m)} imply
    \begin{equation}\label{Eq-L2gamma+2PhimPhi0}
        \frac{\lambda}{2\gamma+2}\int_{\R^d}|\phi_m|^{2\gamma+2}\d x \geqslant \frac{\lambda (2m)^{\gamma+1}}{2\gamma+2}\int_{\R^d}|\varphi_0|^{2\gamma+2}\d x
    \end{equation}
    and
    \begin{equation}
        (2m)^\gamma\lambda\int_{\R^d}|\varphi_0|^{2\gamma+2}\d x \leqslant  \omega_{\phi_m}+\mu_0 \leqslant \frac{\lambda}{2m}\int_{\R^d}|\phi_m|^{2\gamma+2}\d x.\label{Eq-BoundOmegaPhi}
    \end{equation}
    We deduce immediately $\omega_{\phi_m}+\mu_0$ has the same sign as $\lambda$. By \cref{Lem-ControlD12Energy}, there exists $C,\alpha>0$ such that
    $$|\phi_m|_{\D^{1,2}}^2\leqslant C\left(E(m)+ m + \lambda_+ m^{1+\alpha}\right),$$
    thus \cref{Eq-UpBoundE(m)} implies $|\phi_m|_{\D^{1,2}}^2=O(m)$ as $m$ goes to $0$. From the continuous embedding of $\D^{1,2}$ in $\L^{2\gamma+2}(\R^d)$ and \cref{Eq-BoundOmegaPhi}, we deduce $\omega_{\phi_m}+\mu_0=O(m^\gamma)$ as $m$ goes to $0$, for any $\lambda\in\R$. Moreover, for $\lambda>0$, it follows directly from \cref{Eq-BoundOmegaPhi} that $\omega_{\phi_m}$ goes to $+\infty$ as $m$ goes to infinity. Finally, for $\lambda< 0$, using \cref{Lem-InterMassEnergNonlin}, it holds 
    $$m\leqslant \langle (-A+\omega_{\phi_m})\phi_m,\phi_m\rangle + C|\phi_m|_{\L^{2\gamma+2}}^2.$$
    As $\phi_m$ solves \cref{Eq-StationaryAGP}, it holds
    $$\langle (-A+\omega_{\phi_m})\phi_m,\phi_m\rangle = \lambda \int_{\R^d}|\phi_m(x)|^{2\gamma+2}\d x<0.$$
    Hence,
    $$|\phi_m|_{\L^{2\gamma+2}}^2\gtrsim m$$
    and thus, by \cref{Eq-L2gamma+2PhimPhi0},
    \begin{equation}\label{Eq-BoundPhim2gamma+2}
        \frac{1}{m}\int_{\R^d}|\phi_m|^{2\gamma+2}\d x \approx m^\gamma.
    \end{equation}
    In particular, \cref{Eq-BoundOmegaPhi} implies $\omega_{\phi_m}$ goes to $-\infty$ as $m$ goes to infinity.
 
\end{proof}

Using the spectral gap of $-A$, one can deduce the small mass asymptotic for energy ground-states.

\begin{cor}\label[cor]{Cor-SmallMassEnergyGS}
    Assume $\lambda<0$ or $\gamma<\frac{2}{d}$. For any $m>0$, let $\phi_m$ be a non-negative ground-state in $\displaystyle\argmin_{u\in\D^{1,2},\;\mathcal{M}(u)=m}\mathcal{E}(u)$ and define $\psi_m = (2m)^{-\frac{1}{2}}\phi_m$. Then, $\psi_m$ converges to $\varphi_0$ in $\mathrm{D}(A)$ as $m$ goes to $0$.
\end{cor}

\begin{proof}
    Remark $|\psi_m|_{\L^2}=1$ and, by \cref{Lem-ControlD12Energy,Eq-UpBoundE(m)}, $|\psi_m|_{\D^{1,2}}\lesssim 1+m^\alpha$ for some $\alpha>0$. Thus, there exists a sequence $(m_n)_{n\in\N}$ going to 0 and a  $\psi\in\D^{1,2}$ such that $\psi_{m_n}$ converges to $\psi$, weakly in $\D^{1,2}$ and strongly in $\L^2(\R^d)\cap\L^{2\gamma+2}(\R^d)$. Thus, $-A\psi_{m_n}$ converges weakly to $-A\psi$ in $\D^{-1,2}$ and, by \cref{Prop-FrequencyEnergyGS}, $\omega_{\psi_{m_n}}\psi_{m_n}$ converges strongly to $\mu_0\psi$ in $\L^2(\R^d)$. As $\psi_n$ solves
    $$-A\psi_n-\lambda(2m_n)^\gamma|\psi_n|^{2\gamma}\psi_n+\omega_{\psi_{n}}\psi_n=0,$$
    and $A$ is closed, it implies $-A\psi=\mu_0\psi$. Thus, $\psi=\alpha\varphi_0$ ($\alpha\in\C$), by \cref{Prop-SpectralGapA1d,Prop-SpectralGapA2d}. As $|\psi_{m_n}|_{\L^2}=1$, $|\alpha|=1$ and as $\psi_{m_n}\geqslant 0$ and $\varphi_0$ is continuous and positive, $\alpha=1$. Thus, $\psi_m$ converges to $\varphi_0$, weakly in $\D^{1,2}$ and strongly in $\L^2(\R^d)\cap\L^{2\gamma+2}(\R^d)$ as $m$ goes to 0. Finally,
    $$|A\psi_{m}-A\varphi_0|_{\L^2} \leqslant |\lambda|(2m)^{\gamma}|\psi_m|_{\L^{4\gamma+2}}^{2\gamma+1} + |\omega_{\phi_m}+\mu_0| + |\mu_0| |\psi_{m}-\varphi_0|_{\L^2},$$
    implying $\psi_m$ converges strongly to $\varphi_0$ in $\mathrm{D}(A)$.
    
\end{proof}

As a direct consequence of \cref{Cor-SmallMassEnergyGS}, one can refine the asymptotic behavior of $\omega_{\phi_m}$ for the small masses obtained in \cref{Prop-FrequencyEnergyGS} and show
\begin{equation*}
    \lambda (2m)^\gamma \int_{\R^d} |\varphi_0|^{2\gamma+2}\d x \leqslant \omega_{\phi_m} + \mu_0 \leqslant \lambda (2m)^\gamma \int_{\R^d} |\varphi_0|^{2\gamma+2}\d x + o\left(m^\gamma\right)
\end{equation*}
as $m$ goes to 0, which improves \cref{Eq-BoundOmegaPhi} in the limit of small masses.\\

In the defocusing case, energy ground-states cover the whole range of possible non-negative solutions, given by \cref{Prop-APrioriConditionExistence}.

\begin{cor}\label[cor]{Cor-C0curveEnergyGS}

    Assume $\lambda<0$. For any $m>0$, let $\phi_m\in\displaystyle\argmin_{u\in\D^{1,2},\;\mathcal{M}(u)=m}\mathcal{E}(u)$ be the non-negative ground-state and denote by $\omega_m$ its frequency. Then, the map $$m\in(0,+\infty)\mapsto(\omega_m,\phi_m)\in\R\times \mathrm{D}(A)$$ is continuous. In particular, $m\in(0,+\infty)\mapsto \omega_m\in(-\infty,-\mu_0)$ is onto.
    
\end{cor}

\begin{proof}

    Let $m>0$ and $(m_n)_{n\in\N}$ converging to $m$. The proof of \cref{Lem-ContinuityE(m)} implies $\left(\sqrt{\frac{m_n}{m}}\phi_{m_n}\right)_{n\in\N}$ is a minimizing sequence for $E(m)$. By \cref{Prop-ExistenceEnergyGS}, we deduce $\sqrt{\frac{m_n}{m}}\phi_{m_n}$ converges strongly to $\phi_m$ in $\D^{1,2}$, thus implying $m\mapsto \phi_m\in\D^{1,2}$ is continuous. Hence, the map $m\mapsto \omega_m$ is continuous by \cref{Sobolev-embeddings,Eq-Omega(phi),Lem-ContinuityE(m)}. Then, \cref{Prop-FrequencyEnergyGS} and continuity imply the map $m\in(0,+\infty)\mapsto \omega_m\in(-\infty,-\mu_0)$ is onto. Finally, it holds
    $$|A\phi_m-A\phi_{m'}|_{\L^2} \leqslant |\lambda| \left||\phi_m|^{2\gamma}\phi_{m}-|\phi_{m'}|^{2\gamma}\phi_{m'}\right|_{\L^2} + m'|\omega_m-\omega_{m'}|+|\omega_m||\phi_m-\phi_{m'}|_{\L^2}$$
    and the right hand side goes to 0 as $m'$ goes to $m$, as $\D^{1,2}$ is continuously embedded in $\L^2(\R^d)\cap\L^{4\gamma+2}(\R^d).$
    
\end{proof}

\subsection{The (super-)critical focusing case}

We conclude with the unboundeness of energy in the (super-)critical case. For this we introduce a scaling adapted to the exponential transform. For $\alpha>0$ and $u=\rho v\in\D^{1,2}$, we define
$$\forall x\in\R^d,\; u_\alpha(x) = \alpha^{d/2}\rho(x)v(\alpha x).$$
This scaling preserves the space $\D^{1,2}$ but not the mass. Nevertheless, for $u=\rho v\in\D^{1,2}$, it holds for $q\in[2,+\infty)$
$$\lim_{\alpha\to+\infty}\int_{\R^d}|u_\alpha|^q\d x =\rho(0)^q\int_{\R^d}|v|^q\d x.$$
We will thus take advantage of the following lemmas to obtain the unboundedness of the constrained energy.

\begin{lemme}\label[lemme]{Lem-AsympEnergy}
    Let $(u_n)_{n\in\N}\subset\D^{1,2}$. Assume $\mathcal{M}(u_n)$ converges to $m$ and $|u_n|_{\L^{2\gamma+2}}$ is bounded as $n$ goes to infinity. Then 
    $$\mathcal{E}\left(u_n\right)\sim\mathcal{E}\left(\sqrt{\frac{m}{\mathcal{M}(u_n)}}u_n\right)$$
    as $n$ goes to infinity.
\end{lemme}

\begin{proof}
    Let $m_n=\mathcal{M}(u_n)$, it holds
    \begin{align*}
        \mathcal{E}\left(\sqrt{\frac{m}{m_n}}u_n\right) &= \left(\frac{m}{m_n}\right)\left(\frac{1}{2}\langle -Au_n,u_n\rangle - \left(\frac{m}{m_n}\right)^\gamma \frac{\lambda}{2\gamma+2}\int_{\R^d}|u_n|^{2\gamma+2}\d x\right)\\
        &= \left(\frac{m}{m_n}\right)\left(\mathcal{E}(u_n)+ \left(1 - \left(\frac{m}{m_n}\right)^\gamma\right) \frac{\lambda}{2\gamma+2}\int_{\R^d}|u_n|^{2\gamma+2}\d x\right)\\
        &\sim \mathcal{E}\left(u_n\right)
    \end{align*}
    as $n$ goes to infinity.
    
\end{proof}

\begin{lemme}\label[lemme]{Lem-EstimScalTransf}
    For any $s\in[0,1)$, there exists $C=C(s,\Xi)>0$ such that
    $$\forall \alpha>1,\; \forall u\in\W^{s,2},\; |u_\alpha|_{\W^{s,2}}\leqslant C \alpha^s |u|_{\W^{s,2}}.$$
    Moreover, there exists $C=C(\Xi)>0$ such that
    $$\forall \alpha>1,\; \forall u\in\D^{1,2},\; |u_\alpha|_{\D^{1,2}}\leqslant C \alpha |u|_{\D^{1,2}}.$$
\end{lemme}

\begin{proof}

    Let $u=\rho v\in\L^2(\R^d)$, then
    $$\int_{\R^d}|u_\alpha|^2\d x = \int_{\R^d}|v(x)|^2 \rho(x/\alpha)^2\d x \lesssim \int_{\R^d}|u|^2\d x.$$
    Now, for $u=\rho v\in\D^{1,2}$, 
    \begin{align*}
        |u_\alpha|_{\D^{1,2}}^2&\lesssim \alpha^d\int_{\R^d}|\nabla [v(\alpha x)]|^2+|xv(\alpha x)|^2\d x \\
        &\lesssim \alpha^2 \int_{\R^d}|\nabla v(x)|^2 \d x + \alpha^{-2} \int_{\R^d}|xv(x)|^2\d x \lesssim \alpha^2 |u|_{\D^{1,2}}^2,
    \end{align*}
    The result for $s\in(0,1)$ follows by interpolation as \cref{Lem-ProductRule,Def-Dsq} imply $\D^{s,2}=\W^{s,2}$ for $s\in[0,1)$.
    
\end{proof}

\begin{prop}\label[prop]{Prop-UnboundenessEnergy}
    Assume $\lambda>0$. 
    \begin{itemize}
        \item If $\gamma>\frac{2}{d}$, it holds
        $$\forall m>0,\; \inf_{\Atop{u\in\D^{1,2}}{\mathcal{M}(u)=m}}\mathcal{E}(u) = -\infty.$$
        \item If $\gamma=\frac{2}{d}$, there exists a critical mass $m^*_\lambda(\Xi)>0$ such that 
        $$\forall m<m^*_\lambda(\Xi),\; \inf_{\Atop{u\in\D^{1,2}}{\mathcal{M}(u)=m}}\mathcal{E}(u) > -\infty $$
        and
        $$\forall m>m^*_\lambda(\Xi),\; \inf_{\Atop{u\in\D^{1,2}}{\mathcal{M}(u)=m}}\mathcal{E}(u) = -\infty.$$
        Moreover, $m^*_\lambda(\Xi)=m^*_\lambda$ in dimension 1 and 
        $$\exp\left(\inf Y - \sup Y\right)^4 m^*_\lambda\leqslant m^*_\lambda(\Xi) \leqslant m^*_\lambda$$
        in dimension 2, where $m^*_\lambda$ is the critical mass for the classical NLS equation.
    \end{itemize}
\end{prop}

In order to prove \cref{Prop-UnboundenessEnergy} in the case $\gamma=\frac{2}{d}$, we recall the optimal Gagliardo-Nirenberg constant,
\begin{equation}\label{Eq-OptimalGagliardoNirenbergCte}
    J = \inf_{v\in \mathrm{H}^1(\R^d)\backslash\{0\}}\frac{|\nabla v|_{\L^2}^2|v|_{\L^2}^{2\gamma}}{|v|^{2\gamma+2}_{\L^{2\gamma+2}}},
\end{equation}
which is attained by the unique positive and radially symmetric solution $Q$ of
$$\Delta Q + \lambda Q^{2\gamma+1} = Q.$$
It is well-known $Q\in\W^{1,2}$ (see Theorem 8.1.1 in \cite{cazenave2003semilinear}) and Pohozaev identities (see Proposition 1 in \cite{BerestyckiLions}) imply
\begin{equation}\label{Eq-PohozaevQ}
    \frac{1}{2}\int_{\R^d}|\nabla Q(x)|^2\d x = \frac{\lambda}{2\gamma+2}\int_{\R^d}|Q(x)|^{2\gamma+2}\d x.
\end{equation}
Moreover, $\mathcal{M}(Q)=m_\lambda^*=\frac{1}{2}\left(\frac{(\gamma+1)J}{\lambda}\right)^{\frac{1}{\gamma}}$, by \cref{Eq-OptimalGagliardoNirenbergCte,Eq-PohozaevQ}, is the critical mass for the classical NLS equation. Likewise, we introduce the noisy Gagliardo-Nirenberg constant,
\begin{equation}\label{Eq-OptimalNoisyGagliardoNirenbergCte}
    J_\Xi = \inf_{u\in \D^{1,2}\backslash\{0\}}\frac{|u|_{\D^{1,2}}^2|u|_{\L^2}^{2\gamma}}{|u|^{2\gamma+2}_{\L^{2\gamma+2}}}.
\end{equation}

\begin{proof}[Proof of \cref{Prop-UnboundenessEnergy}.]

    \begin{itemize}
        \item Assume $\gamma>\frac{2}{d}$, let $v\in\W^{1,2}\backslash\{0\}$ and $u=\rho v$, then for any $\alpha>0$, \cref{Eq-EnergyVside} implies
        \begin{align*}
            \mathcal{E}(u_\alpha) &= - \frac{\lambda\alpha^{d\gamma}}{2\gamma+2}\int_{\R^d}|v(x)|^{2\gamma+2}\rho(x/\alpha)^{2\gamma+2}\d x &(I)\\
            &+ \frac{\alpha^2}{2}\int_{\R^d}\rho(x/\alpha)^2|\nabla v(x)|^2\d x &(II)\\
            &+ \frac{1}{2\alpha^2}\int_{\R^d} (1-Y(x/\alpha))\rho(x/\alpha)^2|xv(x)|^2\d x &(III)\\
            &-\frac{1}{2}\langle Z,|u_\alpha|^2\rangle. &(IV)
        \end{align*}
        \cref{Lem-EstimScalTransf,Lem-ProductRule} imply $(IV)=O(\alpha^s)$ as $\alpha$ goes to infinity, for any $s\in(0,2)$, as $Z\in\W^{0-,\infty}$. By dominated convergence, $(III)=O(\alpha^{-2})$ as $\alpha$ goes to infinity. Finally, $(II)=O(\alpha^2)$ and
        $$(I)\sim - \frac{\lambda\alpha^{d\gamma}\rho(0)^{2\gamma+2}}{2\gamma+2}\int_{\R^d}|v(x)|^{2\gamma+2}\d x.$$
        Chose $v = \sqrt{2m}\rho(0)^{-1}h_0$ ($m>0$), then $\mathcal{M}(u_\alpha)$ converges to $m$ whereas $\mathcal{E}(u_\alpha)$ goes to $-\infty$, because $d\gamma>2$. Thus, \cref{Lem-AsympEnergy} implies
        $$\forall m>0,\; \inf_{\Atop{u\in\D^{1,2}}{\mathcal{M}(u)=m}}\mathcal{E}(u)=-\infty.$$
        \item Assume $\gamma=\frac{2}{d}$ and let $Q$ be the unique positive and radial solution of 
        $$\Delta Q + \lambda Q^{2\gamma+1} = Q.$$
        Recall $\mathcal{M}(Q)=m_\lambda^*=\frac{1}{2}\left(\frac{(\gamma+1)J}{\lambda}\right)^{\frac{1}{\gamma}}$. For $c>0$, let $u=c\rho Q$. For any $\alpha>0$, it holds
        \begin{align*}
            \mathcal{E}(u_\alpha) &= \frac{\alpha^2}{2}\int_{\R^d}\rho(x/\alpha)^2|c\nabla Q(x)|^2\d x - \frac{\lambda\alpha^{2}}{2\gamma+2}\int_{\R^d}|cQ(x)|^{2\gamma+2}\rho(x/\alpha)^{2\gamma+2}\d x &(I)\\
            &+ \frac{c^2}{2\alpha^2}\int_{\R^d} (1-Y(x/\alpha))\rho(x/\alpha)^2|xQ(x)|^2\d x -\frac{1}{2}\langle Z,|u_\alpha|^2\rangle. &(II)
        \end{align*}
        As previously $(II)=O(\alpha^s)$ as $\alpha$ goes to infinity, for any $s\in(0,2)$, and for $c\rho(0)\neq 1$,
        $$(I)\sim\frac{\alpha^2\lambda(c\rho(0))^2}{2\gamma+2}\left(1-(c\rho(0))^{2\gamma}\right)\int_{\R^d}|Q(x)|^{2\gamma+2}\d x$$
        as $\alpha$ goes to infinity, using \cref{Eq-PohozaevQ}. Thus, for $c\rho(0)>1$, \cref{Lem-AsympEnergy} implies
        $$\inf_{\Atop{u\in\D^{1,2}}{\mathcal{M}(u)=c\rho(0)m^*_\lambda}}\mathcal{E}(u)=-\infty,$$
        that is
        $$\forall m>m^*_\lambda,\; \inf_{\Atop{u\in\D^{1,2}}{\mathcal{M}(u)=m}}\mathcal{E}(u)=-\infty.$$
        Let $m'>m>0$ and $\phi\in\D^{1,2}$ such that $\mathcal{M}(\phi)=m$. Let $\psi=\sqrt{\frac{m'}{m}}\phi$. Then,
        \begin{align*}
            \mathcal{E}(\phi) &= \frac{m}{m'}\left[\frac{1}{2}\langle -A\psi,\psi\rangle - \frac{\lambda}{2\gamma+2}\left(\frac{m}{m'}\right)^{\gamma}\int_{\R^d}|\psi|^{2\gamma+2}\d x\right]\\
            &=\frac{m}{m'}\mathcal{E}(\psi) + \frac{m\lambda}{m'(2\gamma+2)}\left[1 - \left(\frac{m}{m'}\right)^{\gamma}\right]\int_{\R^d}|\phi|^{2\gamma+2}\d x\\
            &\geqslant \frac{m}{m'}\mathcal{E}(\psi).
        \end{align*}
        Thus,
        $$m'\inf_{\Atop{u\in\D^{1,2}}{\mathcal{M}(u)=m}}\mathcal{E}(u) \geqslant m \inf_{\Atop{u\in\D^{1,2}}{\mathcal{M}(u)=m}'}\mathcal{E}(u).$$
        In particular, if $\displaystyle\inf_{u\in\D^{1,2},\; \mathcal{M}(u)=m}\mathcal{E}(u)=-\infty$, then $\displaystyle\inf_{u\in\D^{1,2},\; \mathcal{M}(u)=m'}\mathcal{E}(u)=-\infty$. Similarly, if $\displaystyle\inf_{u\in\D^{1,2},\; \mathcal{M}(u)=m'}\mathcal{E}(u)>-\infty$, then $\displaystyle\inf_{u\in\D^{1,2},\; \mathcal{M}(u)=m}\mathcal{E}(u)>-\infty$. Hence, we can define
        \begin{align*}
            m^*_\lambda(\Xi) &= \inf\left\{m>0,\; \inf_{\Atop{u\in\D^{1,2}}{\mathcal{M}(u)=m}}\mathcal{E}(u)=-\infty\right\}\\
            &= \sup\left\{m>0,\; \inf_{\Atop{u\in\D^{1,2}}{\mathcal{M}(u)=m}}\mathcal{E}(u)>-\infty\right\} \in [0,m^*_\lambda]
        \end{align*}
        To conclude, we only have to prove a lower bound for $m^*_\lambda(\Xi)$. 
        
        Assume $d=1$. By a slight modification of the proof of \cref{Prop-Quasi-coercivity-a-1d}, for every $\eps\in(0,1)$, there exist $\delta_\eps$ such that
        $$\forall u\in\W^{1,2},\; \langle (-A+\delta_\eps)u,u\rangle \geqslant (1-\eps)|u|_{\W^{1,2}}^2.$$
        Thus, for $u\in\D^{1,2}=\W^{1,2}$, it holds 
        $$(1-\eps)|u|_{\W^{1,2}}^2\leqslant 2\left(\mathcal{E}(u)+\delta_\eps \mathcal{M}(u) + \frac{\lambda}{6}\int_{\R^d}|u|^{6}\d x\right)$$
        and
        \cref{Eq-OptimalGagliardoNirenbergCte} implies
        $$\int_{\R}|u|^{6}\d x \leqslant J^{-1}|u|_{\W^{1,2}}^2|u|_{\L^2}^{4}.$$
        Thus, for $m<\frac{1}{2}\left(\frac{3J}{\lambda}\right)^{1/2}=m^*_\lambda$, there exists $C=C(\Xi,m,\lambda)>0$ such that for any $u\in\W^{1,2}$ with $\mathcal{M}(u)=m$,
        $$|u|_{\W^{1,2}}^2\leqslant C\left(1+ \mathcal{E}(u)\right),$$
        showing
        $$\inf_{\Atop{u\in\W^{1,2}}{\mathcal{M}(u)=m}}\mathcal{E}(u)>-\infty$$
        and $m^*_\lambda(\Xi)=m^*_\lambda$.

        Finally, for $d=2$, by a slight modification of the proof of \cref{Prop-Quasi-coercivity-a-2d}, for every $\eps\in(0,1)$, there exist $\delta_\eps$ such that
        $$\forall u\in\D^{1,2},\; \langle (-A+\delta_\eps)u,u\rangle \geqslant (1-\eps)|u|_{\D^{1,2}}^2.$$ 
        Thus, for $u\in\D^{1,2}$, it holds
        $$(1-\eps)|u|_{\D^{1,2}}^2\leqslant 2\left(\mathcal{E}(u)+\delta_\eps \mathcal{M}(u) + \frac{\lambda}{4}\int_{\R^2}|u|^{4}\d x\right).$$
        Now,\cref{Eq-OptimalNoisyGagliardoNirenbergCte} implies
        $$\int_{\R^2}|u|^{4}\d x \leqslant J_\Xi^{-1}|u|_{\D^{1,2}}^2|u|_{\L^2}^{2}.$$
        Thus, for $$m<\frac{J_\Xi}{\lambda},$$ 
        there exists $C=C(\Xi,m,\lambda)>0$ such that for any $u\in\D^{1,2}$ with $\mathcal{M}(u)=m$,
        $$|u|_{\D^{1,2}}^2\leqslant C\left(1+ \mathcal{E}(u)\right),$$
        showing
        $$\inf_{\Atop{u\in\D^{1,2}}{\mathcal{M}(u)=m}}\mathcal{E}(u)>-\infty$$
        and thus $m^*_\lambda(\Xi)\geqslant \frac{J_\Xi}{\lambda}$. For $u=\rho v\in\D^{1,2}$, \cref{Eq-D12nom} implies
        $$\frac{|u|_{\D^{1,2}}^2|u|_{\L^2}^{2}}{|u|^{4}_{\L^{4}}} \geqslant \frac{|\rho\nabla v|_{\L^2}^2|\rho v|_{\L^2}^{2}}{|\rho v|^{4}_{\L^{4}}}\geqslant \left(\frac{\inf \rho}{\sup \rho}\right)^{4}J=\exp(\inf Y - \sup Y)^{4} J,$$
        as $\rho=\e^Y$ is positive. Thus, 
        $$m^*_\lambda(\Xi)\geqslant \frac{J_\Xi}{\lambda}\geqslant\exp(\inf Y-\sup Y)^4\frac{J}{\lambda}=\exp(\inf Y-\sup Y)^4 m^*_\lambda.$$
    \end{itemize}
\end{proof}

\begin{rem}{~}
    \begin{itemize}
        \item If $d=1$ and $m<m_\lambda^*$ or $d=2$ and $m<\frac{J_\Xi}{\lambda}$, the norm in $\D^{1,2}$ is controlled by the mass and the energy, thus a proof by compactness similar to the one of \cref{Prop-ExistenceEnergyGS} shows there exists an energy ground-state.
        \item In dimension 2, we do not know if $m^*_\lambda(\Xi)=\frac{J_\Xi}{\lambda}$ holds. We strongly believe that the lower bound obtained on $m^*_\lambda(\Xi)$ is not optimal.
        \item The case $d=1$, $\gamma=2$ and $m=m^*_\lambda$ is not known yet, but we believe that
        $$\inf_{\Atop{u\in\W^{1,2}}{\mathcal{M}(u)=m^*_\lambda}}\mathcal{E}(u)$$
        has no minimizer, as it is the case when $\Xi=0$ (i.e.\ in presence of a confining potential only). It would be interesting to see if the method used in Theorem 1 in \cite{GuoSeringerMassConcentration} to deal with the case $d=2$, $\gamma=1$ and $\Xi=0$ can be adapted in that case.
    \end{itemize}
\end{rem}

\section{Action ground-states}\label{Sec-ActionGS}

As expected, energy minimization cannot produce standing waves in the focusing supercritical setting. Another classical way to obtain standing waves is to search for solutions of minimal action. The action functional is defined as
$$\forall u\in\D^{1,2},\; \mathcal{S}_{\omega}(u) = \mathcal{E}(u) + \omega\mathcal{M}(u).$$
When needed, we shall write $\mathcal{S}_{\omega}=\mathcal{S}_{\omega,\lambda}$ to emphasize the value of $\lambda$. From the definition of $\mathcal{S}_{\omega}$, we see standing waves are exactly the non-zero critical points of the action. Assume $\lambda>0$ and $\omega>-\mu_0$. For $\alpha>0$, we have
$$\mathcal{S}_{\omega,\lambda}(\alpha\varphi_0) = \frac{\alpha^2}{2}(\omega+\mu_0) - \frac{\lambda\alpha^{2\gamma+2}}{2\gamma+2}\int_{\R^d}|\varphi_0|^{2\gamma+2}\d x \xrightarrow[\alpha\to+\infty]{}-\infty.$$
Hence we cannot minimize the action globally in the focusing case. It is well-known in this case, that one should minimize the action on the set of its critical points \cite{BerestyckiLions,FukuizumiStabilityInstability,LiuRotatingActionGS}, but in order to do so, one should know a priori existence of such a critical point. To avoid needing this a priori knowledge, one can work on a larger constrained set. If $u$ is a critical point of $\mathcal{S}_{\omega,\lambda}$, then it solves \cref{Eq-StationaryAGP} and thus
$$\mathcal{I}_{\omega,\lambda}(u) = \langle (-A+\omega)u,u\rangle - \lambda\int_{\R^d}|u|^{2\gamma+2}\d x = 0.$$
Thus, we can try to minimize $\mathcal{S}_{\omega,\lambda}$ on $$\mathcal{N}_{\omega,\lambda}=\left\{u\in\D^{1,2}\backslash\{0\},\; \mathcal{I}_{\omega,\lambda}(u) = 0\right\}.$$
Assume $u\in\displaystyle\argmin_{\mathcal{N}_{\omega,\lambda}}\mathcal{S}_{\omega,\lambda}$, then there exists a Lagrange multiplier $\Lambda\in\R$ such that
$$\nabla\mathcal{S}_{\omega,\lambda} u = \Lambda \nabla\mathcal{I}_{\omega,\lambda} u.$$
Taking the bracket against $u$ and using $\mathcal{I}_{\omega,\lambda}(u) = \langle\nabla\mathcal{S}_{\omega,\lambda} u,u \rangle$, it follows
$$\Lambda\langle\nabla\mathcal{I}_{\omega,\lambda} u,u \rangle = \langle\nabla\mathcal{S}_{\omega,\lambda} u,u \rangle = \mathcal{I}_{\omega,\lambda}(u)=0$$
as $u\in\mathcal{N}_{\omega,\lambda}$. As 
$$\nabla\mathcal{I}_{\omega,\lambda} u = 2(-A+\omega)u-(2\gamma+2)\lambda|u|^{2\gamma}u,$$
we get
$$2\gamma\Lambda\lambda \int_{\R^d}|u|^{2\gamma+2}\d x =0,$$
thus $\Lambda=0$, i.e.\ $u$ is a critical point of $\mathcal{S}_{\omega,\lambda}$. Moreover, as any critical point of $\mathcal{S}_{\omega,\lambda}$ is in $\mathcal{N}_{\omega,\lambda}$, $u$ minimizes the action among all critical points.\\

The set $\mathcal{N}_{\omega,\lambda}$ is called the Nehari manifold of $\mathcal{S}_{\omega,\lambda}$ (see for example \cite{SzulkinNehariManifold}) and verifies the following useful characterization.

\begin{lemme}\label[lemme]{Lem-CaracNehariManifold}
    Let $\omega>-\mu_0$ and $\lambda>0$. Then for any $u\in\D^{1,2}\backslash\{0\}$, there exists a unique $t_u>0$ such that $t_u u \in \mathcal{N}_{\omega,\lambda}$ and it is the unique maximizer of
    $$t\in(0,+\infty)\mapsto \mathcal{S}_{\omega,\lambda}(t u)\in\R.$$
\end{lemme}

\begin{proof}
    Let $u\in\D^{1,2}\backslash\{0\}$, for any $t>0$, it holds
    $$\mathcal{I}_{\omega,\lambda}(tu)=t^2\langle (-A+\omega)u,u\rangle - \lambda t^{2\gamma+2}\int_{\R^d}|u|^{2\gamma+2}\d x.$$
    Thus,
    \begin{align*}
        t_u u\in\mathcal{N}_{\omega,\lambda} &\iff \mathcal{I}_{\omega,\lambda}(t_u u) = 0\\
        &\iff \lambda t_u^{2\gamma}\int_{\R^d}|u|^{2\gamma+2}\d x=\langle (-A+\omega)u,u\rangle\\
        &\iff t_u = \left(\frac{\langle (-A+\omega)u,u\rangle}{\lambda \int_{\R^d}|u|^{2\gamma+2}\d x}\right)^{\frac{1}{2\gamma}}.
    \end{align*}
    A simple computation gives
    $$\frac{\d}{\d t}\mathcal{S}_{\omega,\lambda}(t u) = t\langle (-A+\omega)u,u\rangle - \lambda t^{2\gamma+1}\int_{\R^d}|u|^{2\gamma+2}\d x$$
    and
    $$\frac{\d^2}{\d t^2}\mathcal{S}_{\omega,\lambda}(t u) = \langle (-A+\omega)u,u\rangle - (2\gamma+1) \lambda t^{2\gamma}\int_{\R^d}|u|^{2\gamma+2}\d x.$$
    Thus, it follows immediately that $\frac{\d}{\d t}\mathcal{S}_{\omega,\lambda}(t u) = 0$ if and only if $t=t_u$ and in this case
    $$\left(\frac{\d^2}{\d t^2}\mathcal{S}_{\omega,\lambda}(t u)\right)_{|t=t_u} = -2\gamma \langle (-A+\omega)u,u\rangle < 0.$$
    
\end{proof}

Our aim is now to prove \cref{Th-ACtionGS}. Showing directly that the minimization problem
\begin{equation}
    \inf_{\mathcal{N}_{\omega,\lambda}}\mathcal{S}_{\omega,\lambda}\label{Eq-MinAction}
\end{equation}
has non-negative solutions may be a bit technical as it is not straightforward that
$$u\in\argmin_{\mathcal{N}_{\omega,\lambda}}\mathcal{S}_{\omega,\lambda}\Rightarrow  |u|\in\argmin_{\mathcal{N}_{\omega,\lambda}}\mathcal{S}_{\omega,\lambda}.$$
Thus, we will introduce two other, but equivalent, minimization problems which are easier to study. Remark if $u$ solves \cref{Eq-StationaryAGP} for some $\lambda>0$, then $\alpha u$ solves \cref{Eq-StationaryAGP} for $\lambda'=\frac{\lambda}{|\alpha|^{2\gamma}}$. Thus, the parameter $\lambda$ can be seen as a Lagrange multiplier, and solving \cref{Eq-StationaryAGP} for some $\lambda>0$ is equivalent to solving \cref{Eq-StationaryAGP} for all $\lambda>0$. Following an idea of \cite{ROSE1988207}, we define for $u\in\D^{1,2}$
\begin{equation}
    \mathcal{P}_\omega(u)=\langle (-A+\omega) u, u\rangle,\label{Eq-POmega}
\end{equation}
\begin{equation}
    \mathcal{Q}(u) = \int_\R |u|^{2(\gamma+1)}\d x\label{Eq-Q(u)}
\end{equation}
and if $u\neq0,$
\begin{equation}
    \mathcal{J}_\omega(u) = \frac{\mathcal{P}_\omega(u)}{\mathcal{Q}(u)^{1/(\gamma+1)}}.\label{Eq-JOmega}
\end{equation}
Let us stress out $\mathcal{J}_\omega(\alpha u)=\mathcal{J}_\omega(u)$ for any $\alpha\in\C^*$, thus if $u$ minimizes $\mathcal{J}_\omega$, $\alpha u$ minimizes $\mathcal{J}_\omega$ too. Using standard argument of variational calculus, one can show minimizers of both
\begin{equation}
    \inf_{\Atop{u\in\D^{1,2}}{\mathcal{Q}(u)=q}}\mathcal{P}_\omega(u)\label{Eq-MinP/Q=q}
\end{equation}
and
\begin{equation}
    \inf_{\D^{1,2}\backslash\{0\}}\mathcal{J}_\omega\label{Eq-MinJ}
\end{equation}
are solutions of \cref{Eq-StationaryAGP} for some $\lambda>0$. For $\lambda>0$, one can rewrite
$$\mathcal{S}_{\omega,\lambda}(u)=\frac{1}{2}\mathcal{P}_\omega(u)-\frac{\lambda}{2(\gamma+1)}\mathcal{Q}(u)$$
and
\begin{equation}
    \mathcal{N}_{\omega,\lambda} = \left\{u\in\D^{1,2}\backslash\{0\},\; \mathcal{J}_\omega(u) = \lambda \mathcal{Q}(u)^{\gamma/(\gamma+1)} \right\}.\label{Eq-NehariJQ}
\end{equation}
As remarked in \cite{ROSE1988207}, minimizers of \cref{Eq-MinP/Q=q,Eq-MinJ} are essentially the same, that is up to a multiplicative factor, due to the the homogeneity of $\mathcal{J}_\omega$. The following proposition shows minimization problems \cref{Eq-MinAction,Eq-MinP/Q=q,Eq-MinJ} are equivalent, up to a multiplicative factor (see also \cite{LiuRotatingActionGS}).

\begin{prop}\label[prop]{Prop-EquivalentMinPbFoca}

    Let $\omega>-\mu_0$ and $q,q',\lambda,\lambda'>0$. Define $\alpha_1=(q'/q)^{1/2(\gamma+1)}$ and $\alpha_2=(\lambda/\lambda')^{1/2\gamma}$. Then, it holds
    \begin{enumerate}
        \item $\displaystyle\argmin_{\Atop{u\in\D^{1,2}}{\mathcal{Q}(u)=q}'}\mathcal{P}_\omega(u) = \alpha_1 \argmin_{\Atop{u\in\D^{1,2}}{\mathcal{Q}(u)=q}}\mathcal{P}_\omega(u)$,
        \item $\displaystyle\argmin_{\mathcal{N}_{\omega,\lambda'}}\mathcal{S}_{\omega,\lambda'} = \alpha_2\argmin_{\mathcal{N}_{\omega,\lambda}}\mathcal{S}_{\omega,\lambda}$,
        \item $\displaystyle\argmin_{\D^{1,2}\backslash\{0\}}\mathcal{J}_{\omega}=\bigcup_{q>0}\argmin_{\Atop{u\in\D^{1,2}}{\mathcal{Q}(u)=q}}\mathcal{P}_\omega(u)=\bigcup_{\lambda>0}\argmin_{\mathcal{N}_{\omega,\lambda}}\mathcal{S}_{\omega,\lambda}$,
        \item Let $J_\omega = \displaystyle\min_{\D^{1,2}\backslash\{0\}}\mathcal{J}_{\omega}>0$, then$\displaystyle\argmin_{u\in\D^{1,2},\; \mathcal{Q}(u)=1}\mathcal{P}_\omega(u) = \argmin_{\mathcal{N}_{\omega,J_\omega}}\mathcal{S}_{\omega,J_\omega}.$
    \end{enumerate}
    
\end{prop}

\begin{proof}
    \begin{enumerate}
        \item Let $u\in\displaystyle\argmin_{u\in\D^{1,2},\; \mathcal{Q}(u)=q}\mathcal{P}_\omega(u)$ and $v\in\displaystyle\argmin_{u\in\D^{1,2},\; \mathcal{Q}(u)=q'}\mathcal{P}_\omega(u)$. Then, $\mathcal{Q}(\alpha_1 u) = q'$ and $\mathcal{Q}(\alpha_1^{-1} v) = q$. It holds
        $$\mathcal{P}_\omega(v)\leqslant \mathcal{P}_\omega(\alpha_1 u)= \alpha_1^2 \mathcal{P}_\omega(u) \leqslant \alpha_1^2 \mathcal{P}_\omega(\alpha_1^{-1}v)=\mathcal{P}_\omega(v).$$
        We conclude
        $$\argmin_{\Atop{u\in\D^{1,2}}{\mathcal{Q}(u)=q}'}\mathcal{P}_\omega(u) = \alpha_1 \argmin_{\Atop{u\in\D^{1,2}}{\mathcal{Q}(u)=q}}\mathcal{P}_\omega(u)$$
        by symmetry of $q$ and $q'$.
        \item Let $u\in\displaystyle\argmin_{\mathcal{N}_{\omega,\lambda}}\mathcal{S}_{\omega,\lambda}$ and $v\in\displaystyle\argmin_{\mathcal{N}_{\omega,\lambda'}}\mathcal{S}_{\omega,\lambda'}$. Then $\alpha_2 u\in \mathcal{N}_{\omega,\lambda'}$ et $\alpha_2^{-1}v\in\mathcal{N}_{\omega,\lambda}$. It holds
        $$\mathcal{S}_{\omega,\lambda'}(v)\leqslant \mathcal{S}_{\omega,\lambda'}(\alpha_2 u) = \frac{\alpha_2^2}{2} \mathcal{P}_\omega(u) - \frac{\lambda' \alpha_2^{2(\gamma+1)}}{2(\gamma+1)}\mathcal{Q}(u) = \alpha_2^2 \mathcal{S}_{\omega,\lambda}(u).$$
        By symmetry of $\lambda$ and $\lambda'$, it implies
        $$\mathcal{S}_{\omega,\lambda}(u) \leqslant \alpha_2^{-2} \mathcal{S}_{\omega,\lambda'}(v),$$
        which allows to conclude $\mathcal{S}_{\omega,\lambda}(\alpha_2 u)=\mathcal{S}_{\omega,\lambda}(v)$ and thus
        $$\argmin_{\mathcal{N}_{\omega,\lambda'}}\mathcal{S}_{\omega,\lambda'} = \alpha_2\argmin_{\mathcal{N}_{\omega,\lambda}}\mathcal{S}_{\omega,\lambda}$$
        by symmetry of $\lambda$ and $\lambda'$.
        \item Let $q>0$, $u\in\displaystyle\argmin_{u\in\D^{1,2},\; \mathcal{Q}(u)=q}\mathcal{P}_\omega(u)$ and $v\in\D^{1,2}\backslash\{0\}$. There exists $t>0$ such that $\mathcal{Q}(tv)=q$. As $u\in\displaystyle\argmin_{u\in\D^{1,2},\; \mathcal{Q}(u)=q}\mathcal{P}_\omega(u)$, it holds
        $$\mathcal{P}_\omega(u)\leqslant \mathcal{P}_\omega(tv)$$
        and thus
        $$\mathcal{J}_\omega(u)=\frac{\mathcal{P}_\omega(u)}{q^{1/(\gamma+1)}}\leqslant \frac{\mathcal{P}_\omega(tv)}{q^{1/(\gamma+1)}} = \mathcal{J}_\omega(tv)=\mathcal{J}_\omega(v).$$
        This shows 
        $$u\in\argmin_{\D^{1,2}\backslash\{0\}}\mathcal{J}_{\omega}.$$
        
        Conversely, let $u\in\displaystyle\argmin_{\D^{1,2}\backslash\{0\}}\mathcal{J}_{\omega}$ and $q=\mathcal{Q}(u)>0$. Let $v\in\displaystyle\argmin_{u\in\D^{1,2},\; \mathcal{Q}(u)=q}\mathcal{P}_\omega(u)$, it holds
        $$\frac{\mathcal{P}_\omega(u)}{q^{1/(\gamma+1)}} = \mathcal{J}_{\omega}(u) \leqslant \mathcal{J}_{\omega}(v) = \frac{\mathcal{P}_\omega(v)}{q^{1/(\gamma+1)}},$$
        and thus 
        $$u\in\argmin_{\Atop{u\in\D^{1,2}}{\mathcal{Q}(u)=q}}\mathcal{P}_\omega(u),$$
        that is 
        $$\argmin_{\D^{1,2}\backslash\{0\}}\mathcal{J}_{\omega}=\bigcup_{q>0}\argmin_{\Atop{u\in\D^{1,2}}{\mathcal{Q}(u)=q}}\mathcal{P}_\omega(u).$$
        The fact that
        $$\bigcup_{q>0}\argmin_{\Atop{u\in\D^{1,2}}{\mathcal{Q}(u)=q}}\mathcal{P}_\omega(u)=\bigcup_{\lambda>0}\argmin_{\mathcal{N}_{\omega,\lambda}}\mathcal{S}_{\omega,\lambda}$$
        follows from 1., 2. and 4.
        \item Let $v\in\displaystyle\argmin_{\mathcal{N}_{\omega,J_\omega}}\mathcal{S}_{\omega,J_\omega}$. As $v\in\mathcal{N}_{\omega,J_\omega}$, \cref{Eq-NehariJQ} implies $\mathcal{Q}(v)\geqslant 1$. Let $u\in\displaystyle\argmin_{u\in\D^{1,2},\; \mathcal{Q}(u)=1}\mathcal{P}_\omega(u)$. Then, $u$ belongs to $\mathcal{N}_{\omega,J_\omega}$. In fact, as $u$ solves \cref{Eq-StationaryAGP} for some $\lambda_\omega>0$, it holds
        $$\mathcal{P}_\omega(u)=\langle (-A+\omega)u,u\rangle = \lambda_\omega\int_{\R^d}|u|^{2\gamma+2}\d x = \lambda_\omega$$
        and as $\mathcal{Q}(u)=1$, $\lambda_\omega=\mathcal{P}_\omega(u)=\mathcal{J}_\omega(u)=J_\omega$ by 3. Thus,
        $$\mathcal{S}_{\omega,J_\omega}(v)\leqslant \mathcal{S}_{\omega,J_\omega}(u) = \frac{\gamma}{2\gamma+2}J_\omega.$$
        Hence,
        $$\mathcal{J}_\omega(v)\leqslant \frac{\gamma}{\gamma+1}J_\omega+\frac{J_\omega\mathcal{Q}(v)^{\gamma/(\gamma+1)}}{\gamma+1},$$
        as $\mathcal{Q}(v)\geqslant 1$. Using $v\in\mathcal{N}_{\omega,J_\omega}$, it holds $J_\omega\mathcal{Q}(v)^{\gamma/(\gamma+1)}=\mathcal{J}_\omega(v)$. Hence,
        $$J_\omega\leqslant\mathcal{J}_\omega(v)\leqslant \frac{\gamma J_\omega}{\gamma+1}+\frac{\mathcal{J}_\omega(v)}{\gamma+1}.$$
        It follows $\mathcal{J}_\omega(v)=J_\omega$. Using once again $v$ belongs to the Nehari manifold, we get $\mathcal{Q}(v)=1$. Hence, $\mathcal{Q}(u)=\mathcal{Q}(v)$ and $\mathcal{P}_\omega(u)=\mathcal{P}_\omega(v)$, which allows to conclude.
    \end{enumerate}
    
\end{proof}

We can use \cref{Prop-EquivalentMinPbFoca} to solve \cref{Eq-MinAction} and prove \cref{Th-ACtionGS}.

\begin{proof}[Proof of \cref{Th-ACtionGS}.]

    By \cref{Prop-Apriori>0}, it is sufficient to exhibit a non-negative $u\in\displaystyle\argmin_{\mathcal{N}_{\omega,\lambda}}\mathcal{S}_{\omega,\lambda}.$ Using \cref{Prop-EquivalentMinPbFoca}, it is sufficient to prove there exists a non-negative $u$ in $\displaystyle\argmin_{u\in\D^{1,2},\; \mathcal{Q}(u)=1}\mathcal{P}_\omega(u)$. Remark that
    $$u\in\argmin_{\Atop{u\in\D^{1,2}}{\mathcal{Q}(u)=1}}\mathcal{P}_\omega(u)\Rightarrow |u|\in\argmin_{\Atop{u\in\D^{1,2}}{\mathcal{Q}(u)=1}}\mathcal{P}_\omega(u),$$
    thus we only need to prove $\displaystyle\argmin_{u\in\D^{1,2},\; \mathcal{Q}(u)=1}\mathcal{P}_\omega(u)$ is not empty. Let $(u_n)_{n\in\N}\subset\D^{1,2}$ be a minimizing sequence of
    $$\inf_{\Atop{u\in\D^{1,2}}{\mathcal{Q}(u)=1}}\mathcal{P}_\omega(u)=P_\omega.$$
    As $\omega>-\mu_0$, there exists $C>0$ such that
    $$\forall u\in\D^{1,2},\; |u|_{\D^{1,2}}^2\leqslant C\mathcal{P}_\omega( u).$$
    Thus, $(u_n)_{n\in\N}$ is bounded in $\D^{1,2}$ and, up to a subsequence, $u_n$ converges to some $u\in\D^{1,2}$, weakly in $\D^{1,2}$ and strongly in $\L^{2\gamma+2}(\R^d)$. It follows $\mathcal{Q}(u)=1$ and, by sequential weak lower semicontinuity,
    $$P_\omega\leqslant\mathcal{P}_\omega (u)\leqslant\liminf_{n\to+\infty}\mathcal{P}_\omega (u_n)=P_\omega,$$
    thus
    $$u\in\argmin_{\Atop{u\in\D^{1,2}}{\mathcal{Q}(u)=1}}\mathcal{P}_\omega(u).$$
    
\end{proof}

\section*{Acknowledgments}

This work was supported by a public grant from the Fondation Mathématique Jacques Hadamard. The author was supported by the ANR project Smooth ANR-22-CE40-0017. This research was supported by the Japan Society for the Promotion of Science (JSPS) Summer Program 2024 (fellow SP24214). The author want to thank Waseda University for hosting their JSPS Summer Program fellowship and in particular Reika Fukuizumi both for organize the fellowship and for their conversation on standing waves in presence of confining potentials. Finally, the author warmly thanks their PhD advisor, Anne de Bouard, for her constant support and helpful advices during the redaction of this paper.

\printbibliography

\end{document}